\documentclass[a4paper, 11pt]{article}

\usepackage{amsmath, amssymb, amsthm} % AMS series
\usepackage{color}
\usepackage[top=40truemm,bottom=30truemm,left=30truemm,right=30truemm]{geometry}
\usepackage{hyperref}
\usepackage[all]{xy}

\newcommand{\R}{\mathbb{R}}

\newcommand{\ep}{\varepsilon}

\newcommand{\imply}{\mspace{10mu} \Longrightarrow \mspace{10mu}}

\DeclareMathOperator{\AC}{AC}

\DeclareMathOperator{\dom}{dom}

\DeclareMathOperator{\lip}{lip}

\DeclareMathOperator{\Map}{Map}

\theoremstyle{plain}
\newtheorem{theorem}{Theorem}[section]
\newtheorem{lemma}[theorem]{Lemma}
\newtheorem{proposition}[theorem]{Proposition}
\newtheorem{corollary}[theorem]{Corollary}

\theoremstyle{definition}
\newtheorem{definition}[theorem]{Definition}
\newtheorem{notation}{Notation}

\theoremstyle{remark}
\newtheorem{remark}[theorem]{Remark}

\makeatletter

\@addtoreset{equation}{section}
\makeatother

\begin{document}

\title{$C^1$-smooth dependence on initial conditions and delay: \\
spaces of initial histories of Sobolev type, \\ and differentiability of translation in $L^p$}
\author{Junya Nishiguchi\thanks{Mathematical Science Group, Advanced Institute for Materials Research (AIMR), Tohoku University,
2-1-1 Katahira, Aoba-ku, Sendai, 980-8577, Japan}
\footnote{E-mail: \url{junya.nishiguchi.b1@tohoku.ac.jp}}}
\date{}

\maketitle

\begin{abstract}
The objective of this paper is to clarify the relationship between the $C^1$-smooth dependence of solutions
to delay differential equations (DDEs) on initial histories (i.e., initial conditions) and delay parameters.
For this purpose, we consider a class of DDEs which include a constant discrete delay.
The problem of $C^1$-smooth dependence is fundamental from the viewpoint of the theory of differential equations.
However, the above mentioned relationship is not obvious
because the corresponding functional differential equations have the less regularity with respect to the delay parameter.
In this paper, we prove that the $C^1$-smooth dependence on initial histories and delay holds
by adopting spaces of initial histories of Sobolev type,
where the differentiability of translation in $L^p$ plays an important role.

\begin{flushleft}
\textbf{2010 Mathematics Subject Classification}.
Primary: 34K05, Secondary: 46E35.
\end{flushleft}

\begin{flushleft}
\textbf{Keywords}.
Delay differential equations;
Constant discrete delay;
Smooth dependence on delay;
History spaces of Sobolev type;
Differentiability of translation in $L^p$
\end{flushleft}

\end{abstract}

\tableofcontents

\section{Introduction}

Differential equations with constant discrete delays are used
for mathematical models of various dynamic phenomena
(e.g., see \cite[Section 21]{Driver 1977}, \cite[Chapter 2]{Kolmanovskii--Myshkis 1999}, and \cite{Erneux 2009}).
In many cases, the precise values of delays are unknown.
Therefore, it is important to study how the solutions behave as functions of delay parameters
in order to investigate the validity of such mathematical models.
This is known as the delay parameter identification problem
(e.g., see \cite{Hartung 2001} and \cite{Banks--Robbins--Sutton 2013}),
where it is necessary to differentiate solutions to delay differential equations (DDEs) with respect to delay parameters.
Indeed, the above mentioned differentiability problem is fundamental
from the viewpoint of the theory of differential equations.
However, the smoothness of the corresponding retarded functional differential equations (RFDEs)
is closely related to the regularity of initial histories.
Therefore, it is not obvious which spaces of initial histories (called \textit{history spaces} in this paper) should be chosen
in order to obtain such differentiability or, in other words, the $C^1$-smooth dependence on delay.

The objective of this paper is to clarify the connection
between the $C^1$-smooth dependence on initial histories and delay and the regularity of initial histories.
For this purpose, we consider a DDE
	\begin{equation}\label{eq:single const delay II}
		\dot{x}(t) = f(x(t), x(t - r))
	\end{equation}
and its initial value problem (IVP)
	\begin{equation}\label{eq:IVP, single const delay II}
		\left\{
		\begin{alignedat}{2}
			\dot{x}(t) &= f(x(t), x(t - r)), & \mspace{20mu} & t \ge 0, \\
			x(t) &= \phi(t), & & t \in [-R, 0]
		\end{alignedat}
		\right.
	\end{equation}
for each $(\phi, r) \in C([-R, 0], \mathbb{R}^N) \times [0, R]$.
Here $R > 0$ is the maximal delay which is constant, $r \in [0, R]$ is the delay parameter,
$N \ge 1$ is an integer, and $f \colon \mathbb{R}^N \times \mathbb{R}^N \to \mathbb{R}^N$ is a function.
$C([-R, 0], \mathbb{R}^N)$ denotes the Banach space of continuous functions
from $[-R, 0]$ to $\mathbb{R}^N$ with the supremum norm
	\begin{equation*}
		\|\phi\|_{C[-R, 0]} := \sup_{\theta \in [-R, 0]} |\phi(\theta)|,
	\end{equation*}
where $|\cdot|$ is a norm on $\mathbb{R}^N$.
Under the local Lipschitz continuity of $f$, \eqref{eq:IVP, single const delay II} has the unique maximal solution
	\begin{equation*}
		x(\cdot; \phi, r) \colon [-R, T_{\phi, r}) \to \mathbb{R}^N
	\end{equation*}
for $0 < T_{\phi, r} \le \infty$.
We refer the reader to \cite{Hale--Lunel 1993} as a general reference of the theory of RFDEs.
Then the problem of the $C^1$-smooth dependence on initial histories and delay which will be studied in this paper
is the continuous differentiability of
	\begin{equation*}
		(\phi, r) \mapsto x(\cdot; \phi, r)
	\end{equation*}
in an appropriate sense.

The difficulty about the $C^1$-smooth dependence on delay
is the less smoothness of the corresponding functional $F$ (called \textit{history functional} in this paper) given by
	\begin{equation}\label{eq:history functional}
		F(\phi, r) := f(\phi(0), \phi(-r))
	\end{equation}
with respect to the delay parameter $r$.
In fact, the function $r \mapsto F(\phi, r)$ is not differentiable for general $\phi \in C([-R, 0], \mathbb{R}^N)$
even if the function $f \colon \mathbb{R}^N \times \mathbb{R}^N \to \mathbb{R}^N$ is smooth.
This phenomenon is similar to the lack of smoothness for history functionals corresponding to state-dependent DDEs
(see \cite{Walther 2003c}).
We refer the reader to \cite{Hartung--Krisztin--Walther--Wu 2006}
as a reference of the theory of state-dependent DDEs.

It is natural to consider initial histories with better regularity
in order to obtain the smooth dependence on initial histories and delay. 
The method of consideration in \cite{Walther 2003c}
is to adopt the Banach space $C^1([-R, 0], \mathbb{R}^N)$ of continuously differentiable functions from $[-R, 0]$ to $\mathbb{R}^N$
with the $C^1$-norm
	\begin{equation*}
		\|\phi\|_{C^1[-R, 0]} := \|\phi\|_{C[-R, 0]} + \|\phi'\|_{C[-R, 0]}
	\end{equation*}
as a history space.
Then the compatibility condition given by
	\begin{equation*}
		\phi'(0) = f(\phi(0), \phi(-r))
	\end{equation*}
for every initial history $\phi$ is necessary to keep the histories of solution of class $C^1$,
and therefore, the solution manifold defined by
	\begin{equation*}
		X_{f, r} = \{\mspace{2mu} \phi \in C^1([-R, 0], \mathbb{R}^N) : \phi'(0) = f(\phi(0), \phi(-r)) \mspace{2mu}\}
	\end{equation*}
arises as the set of initial histories.
However, the framework of the solution manifold is not suitable for the $C^1$-smooth dependence on delay
because $X_{f, r}$ depends on $r$.

The first study of the $C^1$-smooth dependence on initial histories and delay seems to be done
by Hale \& Ladeira~\cite{Hale--Ladeira 1991}.
Their idea is to use the history space $C^{0, 1}([-R, 0], \mathbb{R}^N)$ endowed with the $\mathcal{W}^{1, 1}$-norm.
Here $C^{0, 1}([-R, 0], \mathbb{R}^N)$ denotes the set of Lipschitz continuous functions from $[-R, 0]$ to $\mathbb{R}^N$,
and $\mathcal{W}^{1, p}$-norm for $1 \le p < \infty$ is defined as follows for absolutely continuous functions:
	\begin{equation*}
		\|\phi\|_{\mathcal{W}^{1, p}[-R, 0]}
		:= \biggl( |\phi(-R)|^p + \int_{-R}^0 |\phi'(\theta)|^p \mspace{2mu} \mathrm{d}\theta \biggr)^{\frac{1}{p}}
	\end{equation*}
with the almost everywhere derivative $\phi'$ of $\phi$.
The contribution in \cite{Hale--Ladeira 1991} is the adoption of the Lipschitz continuous regularity
for the $C^1$-smooth dependence on delay.
In this case, the $C^1$-smooth dependence on delay is not trivial
because the history functional given in \eqref{eq:history functional} is not differentiable with respect to $r$
for general $\phi \in C^{0, 1}([-R, 0], \mathbb{R}^N)$.
It should be noticed that the differentiability of $r \mapsto x(\cdot; \phi, r)$ at $r = 0$ is not discussed
in \cite{Hale--Ladeira 1991}.
The continuous differentiability of
	\begin{equation*}
		r \mapsto x(t; \phi, r) \in \mathbb{R}^N
	\end{equation*}
for the time-dependent delay function $r = r(\cdot)$ is studied by Hartung~\cite{Hartung 2016b}
by assuming $\phi \in C^{0, 1}([-R, 0], \mathbb{R}^N)$, where the positivity $r(t) > 0$ is also assumed.

The method of the proof of the $C^1$-smooth dependence on initial histories and delay given in \cite{Hale--Ladeira 1991}
is the fixed point argument, which is standard in the literature (ref.\ \cite{Hale--Lunel 1993}).
That is, IVP~\eqref{eq:IVP, single const delay II} is converted to the fixed point problem
through the integral equation.
Then the $C^1$-smooth dependence on initial histories and delay is obtained
from the $C^1$-uniform contraction theorem (e.g., see \cite[Theorem 2.2 in Chapter 2]{Chow--Hale 1982}),
where history and delay are parameters.
However, the history space
	\begin{equation*}
		\left( C^{0, 1}([-R, 0], \mathbb{R}^N), \|\cdot\|_{\mathcal{W}^{1, 1}[-R, 0]} \right)
	\end{equation*}
is not a Banach space but a \textit{quasi-Banach space} in their terminology.
Therefore, the usual $C^1$-uniform contraction theorem cannot be applicable,
and it is necessary to invent the $C^1$-uniform contraction theorem for such quasi-Banach spaces
(\cite[Theorem 2.7]{Hale--Ladeira 1991}).
It should be noticed that the Banach space $C^{0, 1}([-R, 0], \mathbb{R}^N)$ endowed with the $C^{0, 1}$-norm
	\begin{equation*}
		\|\phi\|_{C^{0, 1}[-R, 0]} := \max\{\|\phi\|_{C[-R, 0]}, \lip(\phi)\},
	\end{equation*}
where $\lip(\phi)$ is the Lipschitz constant of $\phi$, is not suitable for a history space
(see \cite{Louihi--Hbid--Arino 2002}).

Hale \& Ladeira~\cite{Hale--Ladeira 1991} gives an insight into the $C^1$-smooth dependence problem
as mentioned above.
However, the following questions which are related each other should arise:
\begin{itemize}
\item What is the essentiality of the Lipschitz continuous regularity
for the $C^1$-smooth dependence on initial histories and delay?
\item What happens if $\mathcal{W}^{1, p}([-R, 0], \mathbb{R}^N)$ ($1 \le p < \infty$) is chosen as a history space?
\end{itemize}
Here $\mathcal{W}^{1, p}([-R, 0], \mathbb{R}^N)$, which will be called a \textit{history space of Sobolev type} in this paper,
is the linear space of absolutely continuous functions
from $[-R, 0]$ to $\mathbb{R}^N$ whose almost everywhere derivatives belong to $L^p([-R, 0], \mathbb{R}^N)$
endowed with the $\mathcal{W}^{1, p}$-norm.
When $p = 2$ and the norm $|\cdot|$ on $\mathbb{R}^N$ is the Euclidean norm,
$\mathcal{W}^{1, 2}([-R, 0], \mathbb{R}^N)$ becomes a Hilbert space.
This is an advantageous fact for numerical analysis.

In this paper, we show that $\mathcal{W}^{1, p}([-R, 0], \mathbb{R}^N)$ can be chosen as a history space
for the $C^1$-smooth dependence on initial histories and delay.
It becomes clear that the \textit{differentiability of translation in $L^p$} plays an important role in the proof.
The method of the proof is standard but does not require the Lipschitz continuity of initial histories,
which in fact make the proof simple.
This is the reason why $\mathcal{W}^{1, p}([-R, 0], \mathbb{R}^N)$ is appropriate and gives answer to the above questions.
We also prove that the solution semiflow with a delay parameter,
which is the solution semiflow generated by the IVPs of the extended system
	\begin{equation}\label{eq:single const delay II, extended}
		\left\{
		\begin{alignedat}{2}
			\dot{x}(t) &= f(x(t), x(t - r(t))), \\
			\dot{r}(t) &= 0,
		\end{alignedat}
		\right.
	\end{equation}
is a $C^1$-maximal semiflow.
We note that the extended system~\eqref{eq:single const delay II, extended} is a special case
of the following coupled system of DDE and ODE
(see \cite{Arino--Hadeler--Hbid 1998} and \cite{Chen--Hu--Wu 2010})
	\begin{equation*}
		\left\{
		\begin{alignedat}{2}
			\dot{x}(t) &= f(x(t), x(t - r(t))), \\
			\dot{r}(t) &= g(x(t), r(t)),
		\end{alignedat}
		\right.
	\end{equation*}
where $g \colon \mathbb{R}^N \times \mathbb{R} \to \mathbb{R}$ is a function.
The extended system~\eqref{eq:single const delay II, extended} also appears in bifurcation problems
(ref.\ \cite{Kuznetsov 2004}).

Finally, we give another several comments about previous studies.
(i) In \cite{Hale--Ladeira 1991}, the function $f$ is required to be of class $C^2$
for the $C^1$-smooth dependence on initial histories and delay.
The results which will be given in this paper require that $f$ is of class $C^1$ for such $C^1$-smooth dependence,
which is same as \cite{Hartung 2016b}.
(ii) It is mentioned in \cite[Section 4]{Hale--Ladeira 1991} that similar results hold with the same proofs
when the delay is time-dependent.
However, this is incorrect because a simple counter example can be given as follows:
We consider the function $f(x, y) = y$.
Let $0 < T < R$.
For each $c \in [0, R - T]$, we define $r_c \in C(\mathbb{R}, [0, R])$ by
	\begin{equation*}
		r_c(t) =
		\begin{cases}
			c & (t \le 0), \\
			t + c & (0 \le t \le T), \\
			T + c & (t \ge T).
		\end{cases}
	\end{equation*}
Then it can be shown that $[0, R - T] \ni c \mapsto r_c \in C(\mathbb{R}, [0, R])$ is differentiable but the solution
	\begin{equation*}
		x(t; \phi, r_c)
		= \phi(0) + \int_0^t \phi(s - r_c(s)) \mspace{2mu} \mathrm{d}s
		= \phi(0) + t\phi(-c)
		\mspace{20mu} (\forall t \in [0, T])
	\end{equation*}
is not differentiable with respect to $c$ for general $\phi \in C^{0, 1}([-R, 0], \mathbb{R}^N)$.
This example can be considered to be a critical case
in the sense that the \textit{delayed argument function}
	\begin{equation*}
		t \mapsto t - r_c(t)
	\end{equation*}
is constant.
In \cite{Hartung 2016b}, the $C^1$-smooth dependence on time-dependent delay
with the Lipschitz continuous regularity of initial histories is studied
under some strict monotonicity condition of the delayed argument function.
See also \cite{Hartung--Turi 1997} for state-dependent DDEs.
(iii) In \cite{Banks--Robbins--Sutton 2013}, the authors study the $C^1$-smoothness of the function
	\begin{equation*}
		(0, R] \ni r \mapsto x(t; \phi, r) \in \mathbb{R}^N
	\end{equation*}
without citing the previous studies.
It seems that the argument relies on the differentiability of translation in $L^2$
assuming initial histories belong to $H^{1, \infty}$,
however, the proof of the differentiability and the definition of $H^{1, \infty}$ are not given.
The assumption given in \cite{Banks--Robbins--Sutton 2013} is also more stronger,
namely, the boundedness of the norm of the Fr\'{e}chet derivative of $f$ is assumed.

This paper is organized as follows.
In Section~\ref{sec:history spaces of Sobolev type},
we define history spaces of Sobolev type and investigate their fundamental properties.
Section~\ref{sec:main results} are divided into two parts:
a simple case (Subsection~\ref{subsec:a special case}) and a general case (Subsection~\ref{subsec:a general case}).
In Subsection~\ref{subsec:a special case}, we concentrate our consideration on a DDE
	\begin{equation}\label{eq:single const delay I}
		\dot{x}(t) = f(x(t - r))
	\end{equation}
and its IVP
	\begin{equation}\label{eq:IVP, single const delay I}
		\left\{
		\begin{alignedat}{2}
			\dot{x}(t) &= f(x(t - r)), & \mspace{20mu} & t \ge 0, \\
			x(t) &= \phi(t), & & t \in [-R, 0]
		\end{alignedat}
		\right.
	\end{equation}
for each $(\phi, r) \in C([-R, 0], \mathbb{R}^N) \times (0, R]$.
Then the problem of the $C^1$-smooth dependence on delay is very simplified,
and the result directly follows by the continuity and differentiability of translation in $L^p$.
In Subsection~\ref{subsec:a general case},
we consider a general class of DDEs of the form given in \eqref{eq:single const delay II}.
Here we prove the main results of this paper,
which consist of the $C^1$-smooth dependence on initial histories and delay
(Theorem~\ref{thm:C^1-smooth dependence on history and delay, single const delay})
and the $C^1$-smoothness of the solution semiflow with a delay parameter
(Theorem~\ref{thm:C^1-smoothness of sol semiflow with delay, single const delay}).
As mentioned above, the differentiability of translation in $L^p$ plays an important role in the proof.

We have three appendices.
In Appendix~\ref{sec:differentiability of translation in L^p},
we give a proof of this differentiability result (Corollary~\ref{cor:differentiability of translation in L^p})
together with the discussion about the estimate of the double integral for the translation of $L^p$-functions
(Corollary~\ref{cor:double integral, small order, L^p}).
The latter is also used in the proof of the $C^1$-smooth dependence result.
In Appendix~\ref{sec:C^1-uniform contraction theorem},
the $C^1$-uniform contraction theorem (Theorem~\ref{thm:C^1-uniform contraction thm})
is stated with the assumption that the parameter set is not necessarily open.
We give the proof and some fundamental properties about Fr\'{e}chet differentiability
to keep this paper self-contained.
In Appendix ~\ref{sec:maximal semiflows},
we give definitions about maximal semiflows and prove the theorem (Theorem~\ref{thm:C^1-maximal semiflow})
which ensures that a maximal semiflow is of class $C^1$.

\section{Preliminary: History spaces of Sobolev type}\label{sec:history spaces of Sobolev type}

Let $R > 0$ be a constant and $N \ge 1$ be an integer.
The linear space of all functions from $[-R, 0]$ to $\mathbb{R}^N$ is denoted by $\Map([-R, 0], \mathbb{R}^N)$.
Let $\mathbb{R}_+$ denote the set of all nonnegative real numbers.

\begin{definition}[History]
Let $R > 0$ be a given constant.
Let $a < b$ be real numbers so that $a + R < b$ and $\gamma \colon [a, b] \to \mathbb{R}^N$ be a function.
For every $t \in [a + R, b]$, the function $R_t\gamma \in \Map([-R, 0], \mathbb{R}^N)$ defined by
	\begin{equation*}
		R_t\gamma \colon [-R, 0] \ni \theta \mapsto \gamma(t + \theta) \in \mathbb{R}^N
	\end{equation*}
is called the \textit{history} of $\gamma$ at $t$.
\end{definition}

\begin{definition}[History space]
A linear subspace $H \subset \Map([-R, 0], \mathbb{R}^N)$ is called a \textit{history space}
with the past interval $[-R, 0]$
if the topology of $H$ is given so that the linear operations on $H$ are continuous.
\end{definition}

\begin{definition}[Static prolongation]
For each $\phi \in \Map([-R, 0], \mathbb{R}^N)$,
the function $\bar{\phi} \colon [-R, +\infty) \to \mathbb{R}^N$ defined by
	\begin{equation*}
		\bar{\phi}(t) =
		\begin{cases}
			\phi(t) & (t \in [-R, 0]), \\
			\phi(0) & (t \in \mathbb{R}_+)
		\end{cases}
	\end{equation*}
is called the \textit{static prolongation} of $\phi$.
\end{definition}

\begin{definition}[History space of Sobolev type]
Let $1 \le p < \infty$ and $a < b$ be real numbers.
For each absolutely continuous function $x \colon [a, b] \to \mathbb{R}^N$, let
	\begin{equation*}
		\|x\|_{\mathcal{W}^{1, p}[a, b]}
		:= \bigl( |x(a)|^p + \|x'\|_{L^p[a, b]}^p \bigr)^{\frac{1}{p}},
	\end{equation*}
where $x'$ denotes the almost everywhere derivative of $x$.
Let $\mathcal{W}^{1, p}([a, b], \mathbb{R}^N)$ denote the normed space
	\begin{equation*}
		\left\{ \mspace{2mu} x \in \AC([a, b], \mathbb{R}^N) : x' \in L^p([a, b], \mathbb{R}^N) \mspace{2mu} \right\}
	\end{equation*}
endowed with the norm $\|\cdot\|_{\mathcal{W}^{1, p}[a, b]}$.
The history space $\mathcal{W}^{1, p}([-R, 0], \mathbb{R}^N)$ is called the \textit{history space of Sobolev type}.
\end{definition}

\begin{remark}
History spaces of Sobolev type appear for the investigation of neutral delay differential equations.
See \cite{Driver 1965} for $p = 1$ and \cite{Melvin 1972b} for $1 \le p < \infty$ for examples.
\end{remark}

\begin{lemma}\label{lem:equivalent norm in W^{1, p}}
Let $1 \le p < \infty$ and $a < b$ be real numbers.
We define a norm $\|\cdot\|$ on $\mathcal{W}^{1, p}([a, b], \mathbb{R}^N)$ by
	\begin{equation*}
		\|x\| := \|x\|_{C[a, b]} + \|x'\|_{L^p[a, b]}.
	\end{equation*}
Then $\|\cdot\|$ is equivalent to $\|\cdot\|_{\mathcal{W}^{1, p}[a, b]}$.
\end{lemma}

\begin{proof}
Let $x \in \mathcal{W}^{1, p}([a, b], \mathbb{R}^N)$.

\textbf{Step 1.}
By the relationships between $\ell^p$-norms, we have
	\begin{equation*}
		\|x\|_{\mathcal{W}^{1, p}[a, b]}
		= \bigl( |x(a)|^p + \|x'\|_{L^p[a, b]}^p \bigr)^{\frac{1}{p}}
		\le |x(a)| + \|x'\|_{L^p[a, b]}
		\le \|x\|.
	\end{equation*}

\textbf{Step 2.}
By the fundamental theorem of calculus for absolutely continuous functions, we have
	\begin{equation*}
		|x(t)|
		\le |x(a)| + \int_a^t |x'(s)| \mspace{2mu} \mathrm{d}s
		\le |x(a)| + \|x'\|_{L^1[a, b]}
	\end{equation*}
for all $t \in [a, b]$.
This shows
	\begin{equation*}
		\|x\|_{C[a, b]}
		\le |x(a)| + (b - a)^{\frac{1}{q}}\|x'\|_{L^p[a, b]},
	\end{equation*}
where $q$ is the H\"{o}lder conjugate of $p$.
Therefore,
	\begin{align*}
		\|x\|
		&= \|x\|_{C[a, b]} + \|x'\|_{L^p[a, b]} \\
		&\le \bigl[ (b - a)^{\frac{1}{q}} + 1 \bigr] (|x(a)| + \|x'\|_{L^p[a, b]}) \\
		&\le 2^{\frac{1}{q}} \bigl[ (b - a)^{\frac{1}{q}} + 1 \bigr] \|x\|_{\mathcal{W}^{1, p}[a, b]},
	\end{align*}
where the relation between $\ell^p$-norms is used.

By the above steps, the conclusion holds.
\end{proof}

\begin{remark}
Lemma~\ref{lem:equivalent norm in W^{1, p}} means that
a sequence $(x_n)_{n = 1}^\infty$ in $\mathcal{W}^{1, p}([a, b], \mathbb{R}^N)$ converges to $x$
if and only if $x_n \to x$ uniformly and $x_n' \to x'$ in $L^p$.
\end{remark}

\begin{lemma}
Let $1 \le p < \infty$ and $a < b$ be real numbers.
Then $\mathcal{W}^{1, p}([a, b], \mathbb{R}^N)$ is a Banach space.
\end{lemma}

\begin{proof}
Let $(x_n)_{n = 1}^\infty$ be a Cauchy sequence in $\mathcal{W}^{1, p}([a, b], \mathbb{R}^N)$.
Then $(x_n)_{n = 1}^\infty$ is a Cauchy sequence in $C([a, b], \mathbb{R}^N)$,
and $(x'_n)_{n = 1}^\infty$ is a Cauchy sequence in $L^p([a, b], \mathbb{R}^N)$.
Since these spaces are complete, there are $x \in C([a, b], \mathbb{R}^N)$ and $y \in L^p([a, b], \mathbb{R}^N)$ such that
	\begin{equation*}
		\|x - x_n\|_{C[a, b]} \to 0
			\mspace{10mu} \text{and} \mspace{10mu}
		\|y - x'_n\|_{L^p[a, b]} \to 0
	\end{equation*}
as $n \to \infty$.
By the fundamental theorem of calculus for absolutely continuous functions, we have
	\begin{equation*}
		x_n(t) = x_n(a) + \int_a^t x_n'(s) \mspace{2mu} \mathrm{d}s
			\mspace{20mu}
		(t \in [a, b]).
	\end{equation*}
Then by taking the limit as $n \to \infty$, we obtain
	\begin{equation*}
		x(t) = x(a) + \int_a^t y(s) \mspace{2mu} \mathrm{d}s
			\mspace{20mu}
		(x \in [a, b])
	\end{equation*}
because
	\begin{align*}
		\biggl| \int_a^t (y(s) - x_n'(s)) \mspace{2mu} \mathrm{d}s \biggr|
		&\le \|y - x_n'\|_{L^1[a, b]} \\
		&\le (b - a)^{\frac{1}{q}} \|y - x_n'\|_{L^p[a, b]}.
	\end{align*}
Here $q$ is the H\"{o}lder conjugate of $p$.
This shows $x \in \AC([a, b], \mathbb{R}^N)$ and $x' = y \in L^p([a, b], \mathbb{R}^N)$.
Therefore, $(x_n)_{n = 1}^\infty$ converges to $x$ in $\mathcal{W}^{1, p}([a, b], \mathbb{R}^N)$.
\end{proof}

\begin{lemma}\label{lem:continuity of orbit, Sobolev history sp}
Let $1 \le p < \infty$ and $R, T > 0$ be given.
Then for all $x \in \mathcal{W}^{1, p}([-R, T], \mathbb{R}^N)$, the orbit
	\begin{equation*}
		[0, T] \ni t \mapsto R_tx \in \mathcal{W}^{1, p}([-R, 0], \mathbb{R}^N)
	\end{equation*}
is continuous.
\end{lemma}

\begin{proof}
Let $t_0 \in [0, T]$ be fixed.
For all $t \in [0, T]$, we have
	\begin{equation*}
		\|R_tx - R_{t_0}x\|_{\mathcal{W}^{1, p}[-R, 0]}^p
		= |x(t - R) - x(t_0 - R)|^p + \int_{-R}^0 |x'(t + \theta) - x'(t_0 + \theta)|^p \mspace{2mu} \mathrm{d}\theta,
	\end{equation*}
where the right-hand side converges to $0$ as $t \to t_0$
by the continuity of $x$ and by the continuity of translation in $L^p$.
\end{proof}

\begin{lemma}\label{lem:continuity of history operators, Sobolev history sp}
Let $1 \le p < \infty$ and $R, T > 0$ be given.
Then the family of history operators given by
	\begin{equation*}
		\mathcal{W}^{1, p}([-R, T], \mathbb{R}^N) \ni x \mapsto R_tx \in \mathcal{W}^{1, p}([-R, 0], \mathbb{R}^N),
	\end{equation*}
where $t \in [0, T]$, is pointwise equicontinuous.
\end{lemma}

\begin{proof}
It is sufficient to show the equicontinuity at $0$ because the maps are linear.
Let $t \in [0, T]$.
Then for all $x \in \mathcal{W}^{1, p}([-R, T], \mathbb{R}^N)$,
	\begin{align*}
		\|R_tx\|_{C[-R, 0]} + \|(R_tx)'\|_{L^p[-R, 0]}
		&= \sup_{\theta \in [-R, 0]} |x(t + \theta)| 
		+ \biggl( \int_{-R}^0 |x'(t + \theta)|^p \mspace{2mu} \mathrm{d}\theta \biggr)^{\frac{1}{p}} \\
		&\le \|x\|_{C[-R, T]} + \|x'\|_{L^p[-R, T]}.
	\end{align*}
This shows the conclusion.
\end{proof}

\begin{remark}
By the preceding two lemmas,
	\begin{equation*}
		[0, T] \times \mathcal{W}^{1, p}([-R, T], \mathbb{R}^N) \ni (t, x) \mapsto R_tx \in \mathcal{W}^{1, p}([-R, 0], \mathbb{R}^N)
	\end{equation*}
is continuous.
\end{remark}

\begin{lemma}[Continuity of the prolongation operator]
Let $1 \le p < \infty$ and $R, T > 0$ be given.
Then the prolongation operator given by
	\begin{equation*}
		\mathcal{W}^{1, p}([-R, 0], \mathbb{R}^N) \ni \phi \mapsto \bar{\phi}|_{[-R, T]} \in \mathcal{W}^{1, p}([-R, T], \mathbb{R}^N)
	\end{equation*}
is a continuous linear map.
In particular,
	\begin{equation*}
		\bigl\| \bar{\phi}|_{[-R, T]} \bigr\|_{\mathcal{W}^{1, p}[-R, T]} = \|\phi\|_{\mathcal{W}^{1, p}[-R, 0]}
	\end{equation*}
holds for all $\phi \in \mathcal{W}^{1, p}([-R, 0], \mathbb{R}^N)$.
\end{lemma}

\begin{proof}
For every $\phi \in \mathcal{W}^{1, p}([-R, 0], \mathbb{R}^N)$, we have
	\begin{align*}
		\bigl\| \bar{\phi}|_{[-R, T]} \bigr\|_{\mathcal{W}^{1, p}[-R, T]}
		&=
		\biggl(
			\bigl| \bar{\phi}(-R) \bigr|^p
			+ \int_{-R}^T \bigl| \bar{\phi}'(t) \bigr|^p \mspace{2mu} \mathrm{d}t
		\biggr)^{\frac{1}{p}} \\
		&=
		\biggl(
			|\phi(-R)|^p + \int_{-R}^0 |\phi'(\theta)|^p \mspace{2mu} \mathrm{d}t
		\biggr)^{\frac{1}{p}} \\
		&= \|\phi\|_{\mathcal{W}^{1, p}[-R, 0]}.
	\end{align*}
Therefore, the conclusion holds.
\end{proof}

\section{Main results}\label{sec:main results}

In the proofs,
the function space $\mathcal{W}^{1, p}([-R, 0], \mathbb{R}^N)$ is abbreviated as $\mathcal{W}^{1, p}[-R, 0]$.
This is similar to other function spaces.

\subsection{A special case}\label{subsec:a special case}

Let $N \ge 1$ be an integer, $f \colon \mathbb{R}^N \to \mathbb{R}^N$ be a continuous function, and $R > 0$ be a constant.
We consider a DDE~\eqref{eq:single const delay I}
	\begin{equation*}
		\dot{x}(t) = f(x(t - r))
	\end{equation*}
and its IVP~\eqref{eq:IVP, single const delay I}
	\begin{equation*}
		\left\{
		\begin{alignedat}{2}
			\dot{x}(t) &= f(x(t - r)), & \mspace{20mu} & t \ge 0, \\
			x(t) &= \phi(t), & & t \in [-R, 0]
		\end{alignedat}
		\right.
	\end{equation*}
for each $(\phi, r) \in C([-R, 0], \mathbb{R}^N) \times (0, R]$.
The solution $x(\cdot; \phi, r)$ of \eqref{eq:IVP, single const delay I} is expressed by
	\begin{equation*}
		x(t; \phi, r) = \phi(0) + \int_0^t f(\phi(s - r)) \mspace{2mu} \mathrm{d}s
	\end{equation*}
on the interval $[0, r]$, which is continued to $[-R, +\infty)$ by the method of steps.
Let $|\cdot|$ be a norm on $\mathbb{R}^N$.
The operator norm of a linear map $L \colon \mathbb{R}^N \to \mathbb{R}^N$ with respect to the above norm $|\cdot|$
will be denoted by $\|L\|$.

\begin{proposition}
Let $1 \le p < \infty$ and $0 < T < R$ be given.
Let $\phi \in \mathcal{W}^{1, p}([-R, 0], \mathbb{R}^N)$.
If $f$ is of class $C^1$, then
	\begin{equation*}
		[T, R] \ni r \mapsto x(\cdot; \phi, r) \in \mathcal{W}^{1, p}([-R, T], \mathbb{R}^N)
	\end{equation*}
is a continuously differentiable function whose derivative is given by
	\begin{equation*}
		\left( \frac{\partial}{\partial r}x(\cdot; \phi, r) \right)(t)
		= B_{\phi, r}(t)
		:=
		\begin{cases}
			0 & (t \in [-R, 0]), \\
			-\int_0^t (f \circ \phi)'(s - r) \mspace{2mu} \mathrm{d}s & (t \in [0, T])
		\end{cases}
	\end{equation*}
in $\mathcal{W}^{1, p}([-R, T], \mathbb{R}^N)$.
\end{proposition}

\begin{proof}
Since $f$ is locally Lipschitz continuous,
$f \circ \phi \colon [-R, 0] \to \mathbb{R}^N$ is also absolutely continuous.
Then $f \circ \phi$ is differentiable almost everywhere, and
	\begin{equation*}
		(f \circ \phi)'(\theta) = Df(\phi(\theta))\phi'(\theta)
	\end{equation*}
holds for almost all $\theta \in [-R, 0]$.
Therefore,
	\begin{align*}
		\int_{-R}^0 |(f \circ \phi)'(\theta)|^p \mspace{2mu} \mathrm{d}\theta
		&\le \int_{-R}^0 \|Df(\phi(\theta))\|^p |\phi'(\theta)|^p \mspace{2mu} \mathrm{d}\theta \\
		&\le \sup_{\theta \in [-R, 0]} \|Df(\phi(\theta))\|^p \cdot \|\phi'\|_{L^p[-R, 0]}^p,
	\end{align*}
which shows $f \circ \phi \in \mathcal{W}^{1, p}([-R, 0], \mathbb{R}^N)$.

Let $r_0 \in [T, R]$ be fixed.
Then for all $r \in [R, T]$ and all $t \in [-R, T]$,
	\begin{align*}
		&\frac{1}{r - r_0}(x(t; \phi, r) - x(t; \phi, r_0)) - B_{\phi, r_0}(t) \\
		&=
		\begin{cases}
			0 & (t \in [-R, 0]), \\
			\frac{1}{r - r_0} \int_0^t
				\bigl( f(\phi(s - r)) - f(\phi(s - r_0)) + (r - r_0)(f \circ \phi)'(s - r_0) \bigr)
			\mspace{2mu} \mathrm{d}s & (t \in [0, T]).
		\end{cases}
	\end{align*}
Therefore,
	\begin{align*}
		&\left\| \frac{1}{r - r_0}(x(\cdot; \phi, r) - x(\cdot; \phi, r_0)) - B_{\phi, r_0} \right\|_{\mathcal{W}^{1, p}[-R, T]} \\
		&=
		\frac{1}{|r - r_0|} \left( \int_0^T
			|(f \circ \phi)(t - r) - (f \circ \phi)(t - r_0) + (r - r_0) (f \circ \phi)'(t - r_0)|
		\mspace{2mu} \mathrm{d}t \right)^{\frac{1}{p}} \\
		&\to 0
	\end{align*}
as $r \to r_0$ by the differentiability of translation in $L^p$ (Corollary~\ref{cor:differentiability of translation in L^p}).
The continuity of the derivative also holds because
	\begin{align*}
		\|B_{\phi, r} - B_{\phi, r_0}\|_{\mathcal{W}^{1, p}[-R, T]}
		&= \left( \int_0^T |(f \circ \phi)'(t - r) - (f \circ \phi)'(t - r_0)|^p \mspace{2mu} \mathrm{d}t \right)^{\frac{1}{p}} \\
		&\to 0
	\end{align*}
as $r \to r_0$ by the continuity of translation in $L^p$.
\end{proof}

\begin{proposition}
Let $1 \le p < \infty$ and $0 < T < R$ be given.
Suppose that $f$ is of class $C^1$.
Then the family of functions
	\begin{equation*}
		\mathcal{W}^{1, p}([-R, 0], \mathbb{R}^N) \ni \phi \mapsto B_{\phi, r} \in \mathcal{W}^{1, p}([-R, T], \mathbb{R}^N),
	\end{equation*}
where $r \in [T, R]$, is pointwise equicontinuous.
\end{proposition}

\begin{proof}
Let $\phi_0 \in \mathcal{W}^{1, p}([-R, 0], \mathbb{R}^N)$ be fixed and $r \in [T, R]$ be a parameter.
Then we have
	\begin{align*}
		&|(f \circ \phi)'(t - r) - (f \circ \phi_0)'(t - r)| \\
		&\le \|Df(\phi(t - r)) - Df(\phi_0(t - r))\| |\phi'(t - r)| \\
			&\mspace{40mu} + \|Df(\phi_0(t - r))\| |\phi'(t - r) - \phi_0'(t - r)|
	\end{align*}
for all $\phi \in \mathcal{W}^{1, p}([-R, 0], \mathbb{R}^N)$ and all $t \in [0, T]$.
Therefore, by the Minkowski inequality,
	\begin{align*}
		&\|B_{\phi, r} - B_{\phi_0, r}\|_{\mathcal{W}^{1, p}[-R, T]} \\
		&=
		\biggl( \int_0^T
			|(f \circ \phi)'(t - r) - (f \circ \phi_0)'(t - r)|^p
		\mspace{2mu} \mathrm{d}t \biggr)^{\frac{1}{p}} \\
		&\le
		\biggl( \int_0^T
			\|Df(\phi(t - r)) - Df(\phi_0(t - r))\|^p |\phi'(t - r)|^p
		\mspace{2mu} \mathrm{d}t \biggr)^{\frac{1}{p}} \\
		&\mspace{40mu} +
		\biggl( \int_0^T
			\|Df(\phi_0(t - r))\|^p|\phi'(t - r) - \phi_0'(t - r)|^p
		\mspace{2mu} \mathrm{d}t \biggr)^{\frac{1}{p}} \\
		&\le
		\sup_{\theta \in [-R, 0]} \|Df(\phi(\theta)) - Df(\phi_0(\theta))\| \cdot \|\phi\|_{\mathcal{W}^{1, p}[-R, 0]} \\
		&\mspace{40mu} + \sup_{\theta \in [-R, 0]} \|Df(\phi_0(\theta))\| \cdot \|\phi - \phi_0\|_{\mathcal{W}^{1, p}[-R, 0]}.
	\end{align*}
The right-hand side converges to $0$ as $\|\phi - \phi_0\|_{\mathcal{W}^{1, p}[-R, 0]} \to 0$ uniformly in $r$
because $Df$ is uniformly continuous on any closed and bounded set.
\end{proof}

\begin{corollary}
Let $1 \le p < \infty$ and $0 < T < R$ be given.
Suppose that $f$ is of class $C^1$.
Then
	\begin{equation*}
		\mathcal{W}^{1, p}([-R, 0], \mathbb{R}^N) \times [T, R] \ni (\phi, r) \mapsto B_{\phi, r} \in \mathcal{W}^{1, p}([-R, T], \mathbb{R}^N)
	\end{equation*}
is continuous.
\end{corollary}

\begin{proof}
Let $(\phi_0, r_0) \in \mathcal{W}^{1, p}[-R, 0] \times [T, R]$ be fixed.
Then for all $(\phi, r) \in \mathcal{W}^{1, p}[-R, 0] \times [T, R]$, we have
	\begin{align*}
		&\|B_{\phi, r} - B_{\phi_0, r_0}\|_{\mathcal{W}^{1, p}[-R, T]} \\
		&\le \|B_{\phi, r} - B_{\phi_0, r}\|_{\mathcal{W}^{1, p}[-R, T]} + \|B_{\phi_0, r} - B_{\phi_0, r_0}\|_{\mathcal{W}^{1, p}[-R, T]},
	\end{align*}
where the right-hand side converges to $0$ as $(\phi, r) \to (\phi_0, r_0)$ from the above propositions.
\end{proof}

\subsection{A general case}\label{subsec:a general case}

Let $N \ge 1$ be an integer, $f \colon \mathbb{R}^N \times \mathbb{R}^N \to \mathbb{R}^N$ be a continuous function, and $R > 0$ be a constant.
We consider a DDE~\eqref{eq:single const delay II}
	\begin{equation*}
		\dot{x}(t) = f(x(t), x(t - r))
	\end{equation*}
and its IVP~\eqref{eq:IVP, single const delay II}
	\begin{equation*}
		\left\{
		\begin{alignedat}{2}
			\dot{x}(t) &= f(x(t), x(t - r)), & \mspace{20mu} & t \ge 0, \\
			x(t) &= \phi(t), & & t \in [-R, 0]
		\end{alignedat}
		\right.
	\end{equation*}
for each $(\phi, r) \in C([-R, 0], \mathbb{R}^N) \times [0, R]$.
We note that the case $r = 0$ is permitted.
Let $|\cdot|$ be a norm on $\mathbb{R}^N$.
The following product norm on $\mathbb{R}^N \times \mathbb{R}^N$
	\begin{equation*}
		\|(x_1, x_2)\| := |x_1| + |x_2|
	\end{equation*}
will be used.
The operator norms of linear maps $L_1 \colon \mathbb{R}^N \times \mathbb{R}^N \to \mathbb{R}^N$ and $L_2 \colon \mathbb{R}^N \to \mathbb{R}^N$
with respect to the corresponding norms are denoted by $\|L_1\|$ and $\|L_2\|$, respectively.

Let
	\begin{equation*}
		y(t) := x(t) - \bar{\phi}(t) \mspace{20mu} (t \in [0, T])
	\end{equation*}
for some $T > 0$.
Then $x$ is a solution of \eqref{eq:IVP, single const delay II} on $[0, T]$ if and only if $y$ satisfies
	\begin{align*}
		y(t)
		&= \mathcal{T}(y, \phi, r)(t) \\
		&:=
		\begin{cases}
			0 & (t \in [-R, 0]), \\
			\int_0^t f \bigl( (y + \bar{\phi})(s), (y + \bar{\phi})(s - r) \bigr) \mspace{2mu} \mathrm{d}s & (t \in [0, T]).
		\end{cases}
	\end{align*}
The above argument means that $y$ is a fixed point of $\mathcal{T}(\cdot, \phi, r)$
if and only if $x := y + \bar{\phi}$ is a solution of \eqref{eq:IVP, single const delay II}.

For any continuous $y$, $\mathcal{T}(y, \phi, r)$ is absolutely continuous and
	\begin{equation*}
		\|\mathcal{T}(y, \phi, r)\|_{\mathcal{W}^{1, p}[-R, T]}
		=
		\left( \int_0^T
			\bigl| f \bigl( (y + \bar{\phi})(t), (y + \bar{\phi})(t - r) \bigr) \bigr|^p \mspace{2mu}
		\mathrm{d}t \right)^{\frac{1}{p}}
		< \infty
	\end{equation*}
because the integrand is continuous.

\subsubsection{Uniform contraction}

\begin{notation}
Let $T > 0$ be given.
For each $\delta > 0$, let
	\begin{align*}
		\varGamma(\delta)
		:=
		\left\{ \mspace{2mu}
			\gamma \in C([-R, T], \mathbb{R}^N) :
			\text{$R_0\gamma = 0$, $\|\gamma\|_{C[-R, T]} < \delta$}
		\mspace{2mu} \right\}, \\
		\bar{\varGamma}(\delta)
		:=
		\left\{ \mspace{2mu}
			\gamma \in C([-R, T], \mathbb{R}^N) :
			\text{$R_0\gamma = 0$, $\|\gamma\|_{C[-R, T]} \le \delta$}
		\mspace{2mu} \right\},
	\end{align*}
which are considered to be metric spaces with the metric induced by supremum norm.
\end{notation}

\begin{notation}
Let $T > 0$ be given.
For each $1 \le p < \infty$ and $\delta > 0$, let
	\begin{align*}
		\varGamma_{1, p}(\delta)
		:=
		\left\{ \mspace{2mu}
			\gamma \in \mathcal{W}^{1, p}([-R, T], \mathbb{R}^N) :
			\text{$R_0\gamma = 0$, $\|\gamma\|_{\mathcal{W}^{1, p}[-R, T]} < \delta$}
		\mspace{2mu} \right\}, \\
		\bar{\varGamma}_{1, p}(\delta)
		:=
		\left\{ \mspace{2mu}
			\gamma \in \mathcal{W}^{1, p}([-R, T], \mathbb{R}^N) :
			\text{$R_0\gamma = 0$, $\|\gamma\|_{\mathcal{W}^{1, p}[-R, T]} \le \delta$}
		\mspace{2mu} \right\},
	\end{align*}
which are considered to be metric spaces with the metric induced by $\mathcal{W}^{1, p}$-norm.
\end{notation}

\begin{lemma}
Let $1 \le p < \infty$, $B \subset C([-R, 0], \mathbb{R}^N)$ be a bounded set, and $\delta > 0$.
Then for all sufficiently small $T > 0$, the family of maps
	\begin{equation*}
		\mathcal{T}(\cdot, \phi, r) \colon \bar{\varGamma}(\delta) \to \varGamma_{1, p}(\delta),
	\end{equation*}
where $(\phi, r) \in B \times [0, R]$, is well-defined.
\end{lemma}

\begin{proof}
Let $y \in \bar{\varGamma}(\delta)$.
Then for all $(\phi, r) \in B \times [0, R]$,
	\begin{align*}
		\sup_{t \in [0, T]} \bigl\| \bigl( (y + \bar{\phi})(t), (y + \bar{\phi})(t - r) \bigr) \bigr\|
		&= \sup_{t \in [0, T]} (|(y + \bar{\phi})(t)| + |(y + \bar{\phi})(t - r)|) \\
		&\le 2(\|y\|_{C[-R, T]} + \|\phi\|_{C[-R, 0]}).
	\end{align*}
Since $f$ is bounded on any bounded set of $\mathbb{R}^N \times \mathbb{R}^N$, there is $M > 0$ such that
	\begin{equation*}
		\sup_{t \in [0, T]} \bigl| f \bigl( (y + \bar{\phi})(t), (y + \bar{\phi})(t - r) \bigr) \bigr| \le M
	\end{equation*}
for all $(\phi, r) \in B \times [0, R]$.
By choosing $0 < T < (\delta/M)^p$,
	\begin{equation*}
		\|\mathcal{T}(y, \phi, r)\|_{\mathcal{W}^{1, p}[-R, T]}
		\le \left( \int_0^T M^p \mspace{2mu} \mathrm{d}t \right)^{\frac{1}{p}}
		= MT^{\frac{1}{p}}
		< \delta
	\end{equation*}
holds for all such $(\phi, r)$.
This shows the conclusion.
\end{proof}

\begin{lemma}\label{lem:uniform contraction, single const delay}
Let $1 \le p < \infty$, $B \subset C([-R, 0], \mathbb{R}^N)$ be a bounded set, and $\delta > 0$.
If $f$ is locally Lipschitz continuous, then for all sufficiently small $T > 0$, the family of maps
	\begin{equation*}
		\mathcal{T}(\cdot, \phi, r) \colon \bar{\varGamma}(\delta) \to \varGamma_{1, p}(\delta),
	\end{equation*}
where $(\phi, r) \in B \times [0, R]$, is a well-defined uniform contraction.
\end{lemma}

\begin{proof}
The well-definedness follows by the preceding lemma.
Since $f$ is Lipschitz continuous on any bounded set of $\mathbb{R}^N \times \mathbb{R}^N$, there is $L > 0$ such that
	\begin{align*}
		&\bigl| f \bigl( (y_1 + \bar{\phi})(t), (y_1 + \bar{\phi})(t - r) \bigr)
		- f \bigl( (y_2 + \bar{\phi})(t), (y_2 + \bar{\phi})(t - r) \bigr) \bigr| \\
		&\le L (|(y_1 - y_2)(t)| + |(y_1 - y_2)(t - r)|) \\
		&\le 2L \|y_1 - y_2\|_{C[-R, T]}
	\end{align*}
for all $y_1, y_2 \in \bar{\varGamma}(\delta)$, $\phi \in B$, and $r \in [0, R]$.
This implies that we have
	\begin{align*}
		\|\mathcal{T}(y_1, \phi, r) - \mathcal{T}(y_2, \phi, r)\|_{\mathcal{W}^{1, p}[-R, T]}
		&\le \left( \int_0^T (2L \|y_1 - y_2\|_{C[-R, T]})^p \mspace{2mu} \mathrm{d}t \right)^{\frac{1}{p}} \\
		&\le 2LT^{\frac{1}{p}} \cdot \|y_1 - y_2\|_{C[-R, T]}
	\end{align*}
for all such $(y_1, \phi, r)$ and $(y_2, \phi, r)$.
Therefore, the family of maps becomes a well-defined uniform contraction by choosing sufficiently small $0 < T < 1/(2L)^p$.
\end{proof}

\begin{remark}
The uniform contraction means that there is $0 < c < 1$ such that
for all $(\phi, r) \in B \times [0, R]$ and $y_1, y_2 \in \bar{\varGamma}(\delta)$,
	\begin{equation*}
		\|\mathcal{T}(y_1, \phi, r) - \mathcal{T}(y_2, \phi, r)\|_{\mathcal{W}^{1, p}[-R, T]}
		\le c \cdot \|y_1 - y_2\|_{C[-R, T]}
	\end{equation*}
holds.
Therefore, the families of maps
	\begin{align*}
		\mathcal{T}(\cdot, \phi, r) &\colon \bar{\varGamma}(\delta) \to \varGamma(\delta), \\
		\mathcal{T}(\cdot, \phi, r) &\colon \bar{\varGamma}_{1, p}(\delta) \to \varGamma_{1, p}(\delta),
	\end{align*}
where $(\phi, r) \in B \times [0, R]$, are also uniform contractions.
We note that the domains of the two operators correspond to different $T$.
\end{remark}

\begin{proposition}\label{prop:unique existence, single const delay}
Let $B \subset C([-R, 0], \mathbb{R}^N)$ be a bounded set.
If $f$ is locally Lipschitz continuous, then there exists $T > 0$ such that
for every $(\phi, r) \in B \times [0, R]$, IVP~\eqref{eq:IVP, single const delay II} has the unique solution
$x \colon [-R, T] \to \mathbb{R}^N$.
\end{proposition}

\begin{proof}
Let $\delta > 0$ be given.

\textbf{Step 1.}
We choose $M, T > 0$ so that $MT \le \delta$ and
for all $(y, \phi) \in \bar{\varGamma}(\delta) \times B$ and $r \in [0, R]$,
	\begin{equation*}
		\sup_{t \in [0, T]} \bigl| f \bigl( (y + \bar{\phi})(t), (y + \bar{\phi})(t - r) \bigr) \bigr| \le M.
	\end{equation*}
Let $(\phi, r) \in B \times [0, R]$ be given.
Then for every solution $x \colon [-R, T] \to \mathbb{R}^N$ of IVP~\eqref{eq:IVP, single const delay II},
the function $y \colon [-R, T] \to \mathbb{R}^N$ defined by $y = x - \bar{\phi}$
necessarily belongs to $\bar{\varGamma}(\delta)$.

\textbf{Step 2.}
From the preceding lemma, there is sufficiently small $T > 0$ such that the family of maps
	\begin{equation*}
		\mathcal{T}(\cdot, \phi, r) \colon \bar{\varGamma}(\delta) \to \bar{\varGamma}(\delta),
	\end{equation*}
where $(\phi, r) \in B \times [0, R]$, is a uniform contraction.
Then the Banach fixed point theorem implies that for each $(\phi, r) \in B \times [0, R]$,
$\mathcal{T}(\cdot, \phi, r)$ has the unique fixed point $y(\cdot, \phi, r) \in \bar{\varGamma}(\delta)$
because $\bar{\varGamma}(\delta)$ is a complete metric space.
Then $x \colon [-R, T] \to \mathbb{R}^N$ defined by
	\begin{equation*}
		x := y(\cdot, \phi, r) + \bar{\phi}
	\end{equation*}
is a solution of IVP~\eqref{eq:IVP, single const delay II}.
The uniqueness follows by Step 1.
\end{proof}

\begin{remark}
Under the assumption of the local Lipschitz continuity of $f$,
IVP~\eqref{eq:IVP, single const delay II} has the unique maximal solution
	\begin{equation*}
		x(\cdot; \phi, r) \colon [-R, T_{\phi, r}) \to \mathbb{R}^N, \mspace{20mu} \text{where $0 < T_{\phi, r} \le \infty$},
	\end{equation*}
for every $(\phi, r) \in C([-R, 0], \mathbb{R}^N) \times [0, R]$.
\end{remark}

\subsubsection{$C^1$-smoothness with respect to delay}

Let $1 \le p < \infty$ and $T > 0$ be given.
The following notation will be used.

\begin{notation}
For each $(y, \phi) \in C[-R, T] \times C[-R, 0]$, $r \in [0, R]$, and $t \in [0, T]$, let
	\begin{equation*}
		\rho(y, \phi, r, t) := \bigl( (y + \bar{\phi})(t), (y + \bar{\phi})(t - r) \bigr).
	\end{equation*}
Then
	\begin{equation*}
		\rho(y_1, \phi_1, r, t) - \rho(y_2, \phi_2, r, t)
		= \rho(y_1 - y_2, \phi_1 - \phi_2, r, t)
	\end{equation*}
holds.
\end{notation}

\begin{lemma}
Let $(y, \phi) \in C([-R, T], \mathbb{R}^N) \times C([-R, 0], \mathbb{R}^N)$ and $r_0 \in [0, R]$ be fixed.
If $f$ is of class $C^1$, then for $r \in [0, R]$,
	\begin{equation*}
		\sup_{t \in [0, T]} \|Df(\rho(y, \phi, r, t)) - Df(\rho(y, \phi, r_0, t))\|
		\to 0
	\end{equation*}
as $r \to r_0$.
\end{lemma}

\begin{proof}
Since
	\begin{equation*}
		\|\rho(y, \phi, r, t)\| \le 2(\|y\|_{C[-R, T]} + \|\phi\|_{C[-R, 0]})
	\end{equation*}
holds for all $r \in [0, R]$ and $t \in [0, T]$, $\rho(y, \phi, r, t)$ is contained in some bounded set $B$ for all such $r, t$.

Let $\ep > 0$.
The uniform continuity of $Df$ on $B$ implies that
there is $\delta_1 > 0$ such that for all $(x_1, y_1), (x_2, y_2) \in B$,
	\begin{equation*}
		|x_1 - x_2| + |y_1 - y_2| < \delta_1 \imply \|Df(x_1, y_1) - Df(x_2, y_2)\| < \ep.
	\end{equation*}
By the uniform continuity of $y + \bar{\phi} \colon [-R, T] \to \mathbb{R}^N$, there is $\delta_2 > 0$ such that
$|r - r_0| < \delta_2$ implies
	\begin{equation*}
		\sup_{t \in [0, T]} |(y + \bar{\phi})(t - r) - (y + \bar{\phi})(t - r_0)| < \delta_1.
	\end{equation*}
In view of
	\begin{equation*}
		\|\rho(y, \phi, r, t) - \rho(y, \phi, r_0, t)\|
		= |(y + \bar{\phi})(t - r) - (y + \bar{\phi})(t - r_0)|,
	\end{equation*}
the above argument shows that $|r - r_0| < \delta_2$ implies
	\begin{equation*}
		\|Df(\rho(y, \phi, r, t)) - Df(\rho(y, \phi, r_0, t))\| < \ep
	\end{equation*}
for all $t \in [0, T]$.
\end{proof}

\begin{theorem}\label{thm:C^1-smoothness wrt delay}
Let $y \in \mathcal{W}^{1, p}([-R, T], \mathbb{R}^N)$ and $\phi \in \mathcal{W}^{1, p}([-R, 0], \mathbb{R}^N)$ be fixed.
If $f$ is of class $C^1$, then
	\begin{equation*}
		\mathcal{T}(y, \phi, \cdot) \colon [0, R] \to \mathcal{W}^{1, p}([-R, T], \mathbb{R}^N)
	\end{equation*}
is a continuously differentiable function whose derivative is given by
	\begin{align*}
		&\biggl( \frac{\partial}{\partial r} \mathcal{T}(y, \phi, r) \biggr)(t) \\
		&= B_{y, \phi, r}(t) \\
		&:=
		\begin{cases}
			0 & (t \in [-R, 0]), \\
			-\int_0^t
				D_2f \bigl( (y + \bar{\phi})(s), (y + \bar{\phi})(s - r) \bigr) (y + \bar{\phi})'(s - r)
			\mspace{2mu} \mathrm{d}s & (t \in [0, T])
		\end{cases}
	\end{align*}
in $\mathcal{W}^{1, p}([-R, T], \mathbb{R}^N)$.
\end{theorem}

\begin{proof}
\textbf{Step 1.}
Let $r_0 \in [0, R]$ be fixed.
For $y \in \mathcal{W}^{1, p}[-R, T]$ and $\phi \in \mathcal{W}^{1, p}[-R, 0]$, let
	\begin{align*}
		&L(u, t, r) \\
		&:= D_2f \Bigl( (y + \bar{\phi})(t), (y + \bar{\phi})(t - r_0) + u \bigl( (y + \bar{\phi})(t - r) - (y + \bar{\phi})(t - r_0)\bigr) \Bigr)
	\end{align*}
for each $(u, t, r) \in [0, 1] \times [0, T] \times [0, R]$.
We note
	\begin{align*}
		L(0, t, r)
		&:= D_2f \bigl( (y + \bar{\phi})(t), (y + \bar{\phi})(t - r_0) \bigr)
		= D_2f(\rho(y, \phi, r_0, t)), \\
		L(1, t, r)
		&:= D_2f \bigl( (y + \bar{\phi})(t), (y + \bar{\phi})(t - r) \bigr)
		= D_2f(\rho(y, \phi, r, t)).
	\end{align*}
Then
	\begin{align*}
		&f \bigl( (y + \bar{\phi})(t), (y + \bar{\phi})(t - r) \bigr) - f \bigl( (y + \bar{\phi})(t), (y + \bar{\phi})(t - r_0) \bigr) \\
		&= \int_0^1 L(u, t, r) \mspace{2mu} \mathrm{d}u \cdot \bigl( (y + \bar{\phi})(t - r) - (y + \bar{\phi})(t - r_0)\bigr)
	\end{align*}
holds for all $(t, r) \in [0, T] \times [0, R]$.
Therefore, we have
	\begin{align*}
		&\left\|
			\frac{1}{r - r_0} (\mathcal{T}(y, \phi, r) - \mathcal{T}(y, \phi, r_0)) - B_{y, \phi, r_0}
		\right\|_{\mathcal{W}^{1, p}[-R, T]} \\
		&= \frac{1}{|r - r_0|} \biggl( \int_0^T
			\biggl|
				\int_0^1 L(u, t, r) \mspace{2mu} \mathrm{d}u \cdot \bigl( (y + \bar{\phi})(t - r) - (y + \bar{\phi})(t - r_0)\bigr) \\
				&\mspace{40mu} + (r - r_0)L(0, t, r)(y + \bar{\phi})'(t - r_0)
			\biggr|^p
		\mspace{2mu} \mathrm{d}t \biggr)^{\frac{1}{p}} \\
		&=: \frac{1}{|r - r_0|}\biggl( \int_0^T g(t, r)^p \mspace{2mu} \mathrm{d}t \biggr)^{\frac{1}{p}}
	\end{align*}
for all $r \in [0, R]$.

\textbf{Step 2.}
For all $(t, r) \in [0, T] \times [0, R]$,
	\begin{align*}
		&g(t, r) \\
		&\le
		\int_0^1 \|L(u, t, r) - L(0, t, r)\| \mspace{2mu} \mathrm{d}u
		\cdot |(y + \bar{\phi})(t - r) - (y + \bar{\phi})(t - r_0)| \\
			&\mspace{40mu} + \|L(0, t, r)\|
			\cdot |(y + \bar{\phi})(t - r) - (y + \bar{\phi})(t - r_0) + (r - r_0)(y + \bar{\phi})'(t - r_0)| \\
		&\le
		\sup_{(u, t) \in [0, 1] \times [0, T]} \|L(u, t, r) - L(0, t, r)\|
		\cdot |(y + \bar{\phi})(t - r) - (y + \bar{\phi})(t - r_0)| \\
			&\mspace{40mu} + \sup_{t \in [0, T]} \|L(0, t, r)\|
			\cdot |(y + \bar{\phi})(t - r) - (y + \bar{\phi})(t - r_0) + (r - r_0)(y + \bar{\phi})'(t - r_0)| \\
		&=: g_1(t, r) + g_2(t, r).
	\end{align*}
Therefore,
	\begin{align*}
		&\frac{1}{|r - r_0|}\biggl( \int_0^T g(t, r)^p \mspace{2mu} \mathrm{d}t \biggr)^{\frac{1}{p}} \\
		&\le \frac{1}{|r - r_0|} \biggl( \int_0^T g_1(t, r)^p \mspace{2mu} \mathrm{d}t \biggr)^{\frac{1}{p}}
		+ \frac{1}{|r - r_0|} \biggl( \int_0^T g_2(t, r)^p \mspace{2mu} \mathrm{d}t \biggr)^{\frac{1}{p}}
	\end{align*}
by the Minkowski inequality.

\textbf{Step 3.}
Let $\ep > 0$.
In the same way as the preceding lemma,
there is $\delta > 0$ such that for all $r \in [0, R]$, $|r - r_0| < \delta$ implies
	\begin{equation*}
		\sup_{(u, t) \in [0, 1] \times [0, T]} \|L(u, t, r) - L(0, t, r)\| \le \ep
	\end{equation*}
because
	\begin{equation*}
		\Bigl\| \Bigl( 0, u \bigl( (y + \bar{\phi})(t - r) - (y + \bar{\phi})(t - r_0)\bigr) \Bigr) \Bigr\|
		\le |(y + \bar{\phi})(t - r) - (y + \bar{\phi})(t - r_0)|.
	\end{equation*}
Therefore, for such $r$,
	\begin{equation*}
		g_1(t, r)
		\le \ep |(y + \bar{\phi})(t - r) - (y + \bar{\phi})(t - r_0)|
		= \ep \cdot \biggl| \int_{-r_0}^{-r} (y + \bar{\phi})'(t + \theta) \mspace{2mu} \mathrm{d}\theta \biggr|,
	\end{equation*}
and we have
	\begin{align*}
		\frac{1}{|r - r_0|} \biggl( \int_0^T g_1(t, r)^p \mspace{2mu} \mathrm{d}t \biggr)^{\frac{1}{p}}
		&\le \frac{\ep}{|r - r_0|} \biggl( \int_0^T
			\biggl| \int_{-r_0}^{-r} |(y + \bar{\phi})'(t + \theta)| \mspace{2mu} \mathrm{d}\theta \biggr|^p
		\mspace{2mu} \mathrm{d}t \biggr)^{\frac{1}{p}} \\
		&\le \ep \cdot \|y + \bar{\phi}\|_{\mathcal{W}^{1, p}[-R, T]},
	\end{align*}
where the last inequality follows from Corollary~\ref{cor:double integral, small order, L^p}.

\textbf{Step 4.}
For all $r \in [0, R]$, we have
	\begin{align*}
		&\frac{1}{|r - r_0|} \biggl( \int_0^T g_2(t, r)^p \mspace{2mu} \mathrm{d}t \biggr)^{\frac{1}{p}} \\
		&\le \sup_{t \in [0, T]} \|L(0, t, r)\| \\
		&\mspace{40mu} \cdot \frac{1}{|r - r_0|} \biggl( \int_0^T
			|(y + \bar{\phi})(t - r) - (y + \bar{\phi})(t - r_0) + (r - r_0)(y + \bar{\phi})'(t - r_0)|^p
		\mspace{2mu} \mathrm{d}t \biggr)^{\frac{1}{p}},
	\end{align*}
where the last term converges to $0$ by the differentiability of translation in $L^p$
(Corollary~\ref{cor:differentiability of translation in L^p}).

\textbf{Step 5.}
By the above steps, we have
	\begin{align*}
		\left\|
			\frac{1}{r - r_0} (\mathcal{T}(y, \phi, r) - \mathcal{T}(y, \phi, r_0)) - B_{y, \phi, r_0}
		\right\|_{\mathcal{W}^{1, p}[-R, T]}
		\to 0
	\end{align*}
as $r \to r_0$, which shows the Fr\'{e}chet differentiability.
The continuity of the derivative also holds because
	\begin{align*}
		&\|B_{y, \phi, r} - B_{y, \phi, r_0}\|_{\mathcal{W}^{1, p}[-R, T]} \\
		&=
		\biggl( \int_0^T
			|L(1, t, r)(y + \bar{\phi})'(t - r) - L(0, t, r)(y + \bar{\phi})'(t - r_0)|^p
		\mspace{2mu} \mathrm{d}t \biggr)^{\frac{1}{p}} \\
		&\le
		\biggl( \int_0^T
			\|L(1, t, r) - L(0, t, r)\|^p |(y + \bar{\phi})'(t - r)|^p
		\mspace{2mu} \mathrm{d}t \biggr)^{\frac{1}{p}} \\
		&\mspace{40mu} +
		\biggl( \int_0^T
			\|L(0, t, r)\|^p |(y + \bar{\phi})'(t - r) - (y + \bar{\phi})'(t - r_0)|^p
		\mspace{2mu} \mathrm{d}t \biggr)^{\frac{1}{p}} \\
		&\le \sup_{t \in [0, T]} \|L(1, t, r) - L(0, t, r)\| \cdot \|y + \bar{\phi}\|_{\mathcal{W}^{1, p}[-R, T]} \\
		&\mspace{40mu} + \sup_{t \in [0, T]} \|L(0, t, r)\|
		\biggl( \int_0^T
			|(y + \bar{\phi})'(t - r) - (y + \bar{\phi})'(t - r_0)|^p
		\mspace{2mu} \mathrm{d}t \biggr)^{\frac{1}{p}}.
	\end{align*}
This shows that $\|B_{y, \phi, r} - B_{y, \phi, r_0}\|_{\mathcal{W}^{1, p}[-R, T]}$ converges to $0$ as $r \to r_0$
by the preceding lemma and by the continuity of translation in $L^p$.
\end{proof}

\subsubsection{$C^1$-smoothness with respect to prolongation and history}

Let $1 \le p < \infty$ and $T > 0$ be given.

\begin{lemma}
Let $(y_0, \phi_0) \in C([-R, T], \mathbb{R}^N) \times C([-R, 0], \mathbb{R}^N)$ be fixed.
If $f$ is of class $C^1$, then for $(y, \phi) \in C([-R, T], \mathbb{R}^N) \times C([-R, 0], \mathbb{R}^N)$,
	\begin{equation*}
		\sup_{t \in [0, T]} \|Df(\rho(y, \phi, r, t)) - Df(\rho(y_0, \phi_0, r, t))\|
		\to 0
	\end{equation*}
as $(y, \phi) \to (y_0, \phi_0)$ uniformly in $r \in [0, R]$.
\end{lemma}

\begin{proof}
We may assume that there is a bounded set $B \subset \mathbb{R}^N \times \mathbb{R}^N$ such that
	\begin{equation*}
		\rho(y, \phi, r, t) \in B
	\end{equation*}
holds for all $(y, \phi) \in C[-R, T] \times C[-R, 0]$, $r \in [0, R]$, and $t \in [0, T]$ because
	\begin{align*}
		\|\rho(y, \phi, r, t)\|
		&\le \|\rho(y, \phi, r, t) - \rho(y_0, \phi_0, r, t)\| + \|\rho(y_0, \phi_0, r, t)\| \\
		&= \|\rho(y - y_0, \phi - \phi_0, r, t)\| + \|\rho(y_0, \phi_0, r, t)\| \\
		&\le 2(\|y - y_0\|_{C[-R, T]} + \|\phi - \phi_0\|_{C[-R, 0]} + \|y_0\|_{C[-R, T]} + \|\phi_0\|_{C[-R, 0]}).
	\end{align*}

Let $\ep > 0$.
The uniform continuity of $Df$ on $B$ implies that
there is $\delta > 0$ such that for all $(x_1, y_1), (x_2, y_2) \in B$,
	\begin{equation*}
		|x_1 - x_2| + |y_1 - y_2| < \delta \imply \|Df(x_1, y_1) - Df(x_2, y_2)\| < \ep.
	\end{equation*}
Therefore,
	\begin{equation*}
		\|y - y_0\|_{C[-R, T]} + \|\phi - \phi_0\|_{C[-R, 0]} < \frac{\delta}{2}
	\end{equation*}
implies
	\begin{equation*}
		\|Df(\rho(y, \phi, r, t)) - Df(\rho(y_0, \phi_0, r, t))\| < \ep
	\end{equation*}
for all $t \in [0, T]$ uniformly in $r$.
\end{proof}

\begin{theorem}\label{thm:C^1-smoothness wrt prolongation and history}
Let $r \in [0, R]$ be fixed.
If $f$ is of class $C^1$, then
	\begin{equation*}
		\mathcal{T}(\cdot, \cdot, r) \colon C([-R, T], \mathbb{R}^N) \times C([-R, 0], \mathbb{R}^N) \to \mathcal{W}^{1, p}([-R, T], \mathbb{R}^N)
	\end{equation*}
is continuously Fr\'{e}chet differentiable.
The Fr\'{e}chet derivative is given by
	\begin{equation*}
		D_{y, \phi}\mathcal{T}(y, \phi, r) = A_{y, \phi, r},
	\end{equation*}
where
	\begin{align*}
		&[A_{y, \phi, r}(\eta, \chi)](t) \\
		&=
		\begin{cases}
			0 & (t \in [-R, 0]), \\
			\int_0^t
				Df\bigl(
					(y + \bar{\phi})(s), (y + \bar{\phi})(s - r) \bigr) \bigl( (\eta + \bar{\chi})(s), (\eta + \bar{\chi})(s - r)
				\bigr)
			\mspace{2mu} \mathrm{d}s & (t \in [0, T])
		\end{cases}
	\end{align*}
for all $(\eta, \chi) \in C([-R, T], \mathbb{R}^N) \times C([-R, 0], \mathbb{R}^N)$.
In particular,
	\begin{align*}
		&\|A_{y, \phi, r} - A_{y_0, \phi_0, r}\| \\
		&\le
		2T^{\frac{1}{p}} \sup_{t \in [0, T]}
		\bigl\| Df \bigl( (y + \bar{\phi})(t), (y + \bar{\phi})(y - r) \bigr)
		- Df \bigl( (y_0 + \bar{\phi}_0)(t), (y_0 + \bar{\phi}_0)(t - r) \bigr) \bigr\|
	\end{align*}
holds, where $\|\cdot\|$ denotes the corresponding operator norm.
\end{theorem}

\begin{proof}
Let
	\begin{equation*}
		\|(\eta, \chi)\| := \|\eta\|_{C[-R, T]} + \|\chi\|_{C[-R, 0]}
	\end{equation*}
for each $(\eta, \chi) \in C[-R, T] \times C[-R, 0]$.

\textbf{Step 1.}
Let $(y, \phi) \in C[-R, T] \times C[-R, 0]$ be fixed.
Then for all $(\eta, \chi) \in C[-R, T] \times C[-R, 0]$,
	\begin{align*}
		&\|A_{y, \phi, r}(\eta, \chi)\|_{\mathcal{W}^{1, p}[-R, T]} \\
		&=
		\biggl( \int_0^T
			|Df(\rho(y, \phi, r, t))\rho(\eta, \chi, r, t)|^p
		\mspace{2mu} \mathrm{d}t \biggr)^{\frac{1}{p}} \\
		&\le \biggl( 2T^{\frac{1}{p}} \sup_{t \in [0, T]} \|Df(\rho(y, \phi, r, t))\| \biggr) \|(\eta, \chi)\|.
	\end{align*}
This shows that
	\begin{equation*}
		A_{y, \phi, r} \colon C[-R, T] \times C[-R, 0] \to \mathcal{W}^{1, p}[-R, T]
	\end{equation*}
is a bounded linear operator.

\textbf{Step 2.}
Let $(y_0, \phi_0) \in C[-R, T] \times C[-R, 0]$ be fixed.
For $(\eta, \chi) \in C[-R, T] \times C[-R, 0]$, let
	\begin{equation*}
		y := y_0 + \eta \mspace{10mu} \text{and} \mspace{10mu} \phi := \phi_0 + \chi.
	\end{equation*}
Then
	\begin{align*}
		&f(\rho(y, \phi, r, t)) - f(\rho(y_0, \phi_0, r, t)) \\
		&= \int_0^1 Df \bigl( \rho(y_0, \phi_0, r, t) + u\rho(\eta, \chi, r, t) \bigr) \mspace{2mu} \mathrm{d}u
		\cdot \rho(\eta, \chi, r, t)
	\end{align*}
holds for all $r \in [0, R]$ and $t \in [0, T]$.

Let $\ep > 0$.
In the same way as the preceding lemma, there is $\delta > 0$ such that $\|(\eta, \chi)\| \le \delta$ implies
	\begin{equation*}
		\sup_{t \in [0, T]}
		\bigl\| Df \bigl( \rho(y_0, \phi_0, r, t) + u\rho(\eta, \chi, r, t) \bigr) - Df(\rho(y_0, \phi_0, r, t)) \bigr\|
		\le \ep
	\end{equation*}
because
	\begin{equation*}
		\|u\rho(\eta, \chi, r, t)\| \le 2\|(\eta, \chi)\|.
	\end{equation*}
Therefore, for such $(\eta, \chi)$, we have
	\begin{align*}
		&\|\mathcal{T}(y, \phi, r) - \mathcal{T}(y_0, \phi_0, r) - A_{y_0, \phi_0, r}(\eta, \chi)\|_{\mathcal{W}^{1, p}[-R, T]} \\
		&=
		\biggl( \int_0^T
			|f(\rho(y, \phi, r, t)) - f(\rho(y_0, \phi_0, r, t)) - Df(\rho(y_0, \phi_0, r, t))\rho(\eta, \chi, r, t)|^p
		\mspace{2mu} \mathrm{d}t \biggr)^{\frac{1}{p}} \\
		&\le \biggl( \int_0^T
			\biggl( \int_0^1
				\bigl\| Df \bigl( \rho(y_0, \phi_0, r, t) + u\rho(\eta, \chi, r, t) \bigr) - Df(\rho(y_0, \phi_0, r, t)) \bigr\|
			\mspace{2mu} \mathrm{d}u \biggr)^p \\
			&\mspace{40mu} \cdot \|\rho(\eta, \chi, r, t)\|^p
		\mspace{2mu} \mathrm{d}t \biggr)^{\frac{1}{p}} \\
		&\le 2\ep T^{\frac{1}{p}} \|(\eta, \chi)\|.
	\end{align*}
This shows the Fr\'{e}chet differentiability of $\mathcal{T}(\cdot, \cdot, r)$ at $(y_0, \phi_0)$.

\textbf{Step 3.}
Let $(y_0, \phi_0) \in C[-R, T] \times C[-R, 0]$ be fixed.
For all $(y, \phi), (\eta, \chi) \in C[-R, T] \times C[-R, 0]$,
	\begin{align*}
		&\|(A_{y, \phi, r} - A_{y_0, \phi_0, r})(\eta, \chi)\|_{\mathcal{W}^{1, p}[-R, T]} \\
		&=
		\biggl( \int_0^T
			|[Df(\rho(y, \phi, r, t)) - Df(\rho(y_0, \phi_0, r, t))]\rho(\eta, \chi, r, t)|^p
		\mspace{2mu} \mathrm{d}t \biggr)^{\frac{1}{p}} \\
		&\le \biggl( 2T^{\frac{1}{p}} \sup_{t \in [0, T]} \|Df(\rho(y, \phi, r, t)) - Df(\rho(y_0, \phi_0, r, t))\| \biggr)
		\|(\eta, \chi)\|.
	\end{align*}
This shows
	\begin{equation*}
		\|A_{y, \phi, r} - A_{y_0, \phi_0, r}\|
		\le 2T^{\frac{1}{p}} \sup_{t \in [0, T]} \|Df(\rho(y, \phi, r, t)) - Df(\rho(y_0, \phi_0, r, t))\|,
	\end{equation*}
which converges to $0$ uniformly in $r$ as $(y, \phi) \to (y_0, \phi_0)$ by the preceding lemma.
Therefore, $(y, \phi) \mapsto A_{y, \phi, r}$ is continuous at $(y_0, \phi_0)$.

This completes the proof.
\end{proof}

\begin{remark}
The function
	\begin{equation*}
		\mathcal{T}(\cdot, \cdot, r) \colon
		\mathcal{W}^{1, p}([-R, T], \mathbb{R}^N) \times \mathcal{W}^{1, p}([-R, 0], \mathbb{R}^N) \to \mathcal{W}^{1, p}([-R, T], \mathbb{R}^N)
	\end{equation*}
is also continuously Fr\'{e}chet differentiable because the inclusion
	\begin{equation*}
		\mathcal{W}^{1, p}([a, b], \mathbb{R}^N) \subset C([a, b], \mathbb{R}^N), \mspace{20mu} \text{where $a < b$},
	\end{equation*}
is continuous (see Lemma~\ref{lem:equivalent norm in W^{1, p}}).
\end{remark}

\subsubsection{$C^1$-smoothness with respect to prolongation, history, and delay}

Let $1 \le p < \infty$ and $T > 0$ be given.
We continue to use the following notations used
in Theorems~\ref{thm:C^1-smoothness wrt delay} and \ref{thm:C^1-smoothness wrt prolongation and history}.

\begin{notation}
Let $(y, \phi) \in \mathcal{W}^{1, p}[-R, T] \times \mathcal{W}^{1, p}[-R, 0]$ and $r \in [0, R]$.
	\begin{equation*}
		A_{y, \phi, r} \colon \mathcal{W}^{1, p}[-R, T] \times \mathcal{W}^{1, p}[-R, 0] \to \mathcal{W}^{1, p}[-R, T]
	\end{equation*}
is the bounded linear operator defined by
	\begin{equation*}
		[A_{y, \phi, r}(\eta, \chi)](t) =
		\begin{cases}
			0 & (t \in [-R, 0]), \\
			\int_0^t
				Df(\rho(y, \phi, r, s))\rho(\eta, \chi, r, s)
			\mspace{2mu} \mathrm{d}s & (t \in [0, T]).
		\end{cases}
	\end{equation*}
\end{notation}

\begin{notation}
Let $(y, \phi) \in \mathcal{W}^{1, p}[-R, T] \times \mathcal{W}^{1, p}[-R, 0]$ and $r \in [0, R]$.
$B_{y, \phi, r} \in \mathcal{W}^{1, p}[-R, T]$ is defined by
	\begin{equation*}
		B_{y, \phi, r}(t) =
		\begin{cases}
			0 & (t \in [-R, 0]), \\
			-\int_0^t
				D_2f(\rho(y, \phi, r, s)) (y + \bar{\phi})'(s - r)
			\mspace{2mu} \mathrm{d}s & (t \in [0, T]).
		\end{cases}
	\end{equation*}
\end{notation}

\begin{theorem}\label{thm:C^1-smoothness wrt prolongation, history, and delay, single const delay}
Suppose that $f$ is of class $C^1$.
Then 
	\begin{equation*}
		\mathcal{T} \colon \mathcal{W}^{1, p}([-R, T], \mathbb{R}^N) \times \mathcal{W}^{1, p}([-R, 0], \mathbb{R}^N) \times [0, R] \to \mathcal{W}^{1, p}([-R, T], \mathbb{R}^N)
	\end{equation*}
is continuously Fr\'{e}chet differentiable whose Fr\'{e}chet derivative at $(y, \phi, r)$ is given by
	\begin{equation*}
		[D\mathcal{T}(y, \phi, r)](\eta, \chi, \xi)
		= A_{y, \phi, r}(\eta, \chi) + \xi B_{y, \phi, r}
	\end{equation*}
for all $(\eta, \chi, \xi) \in \mathcal{W}^{1, p}([-R, T], \mathbb{R}^N) \times \mathcal{W}^{1, p}([-R, 0], \mathbb{R}^N) \times \mathbb{R}$.
\end{theorem}

\begin{proof}
It is sufficient to show the continuity of
	\begin{equation*}
		(y, \phi, r) \mapsto A_{y, \phi, r}
			\mspace{10mu} \text{and} \mspace{10mu}
		(y, \phi, r) \mapsto B_{y, \phi, r}
	\end{equation*}
with respect to the corresponding operator norms.
Let
	\begin{equation*}
		\|(\eta, \chi)\| := \|\eta\|_{\mathcal{W}^{1, p}[-R, T]} + \|\chi\|_{\mathcal{W}^{1, p}[-R, 0]}
	\end{equation*}
for each $(\eta, \chi) \in \mathcal{W}^{1, p}[-R, T] \times \mathcal{W}^{1, p}[-R, 0]$.

\textbf{Step 1.}
The family of functions
	\begin{equation*}
		(y, \phi) \mapsto A_{y, \phi, r},
	\end{equation*}
where $r \in [0, R]$, is pointwise equicontinuous from Theorem~\ref{thm:C^1-smoothness wrt prolongation and history}.
Therefore, we only have to show the continuity of
	\begin{equation*}
		r \mapsto A_{y, \phi, r}
	\end{equation*}
for each fixed $(y, \phi) \in \mathcal{W}^{1, p}[-R, T] \times \mathcal{W}^{1, p}[-R, 0]$.

Let $r_0 \in [0, R]$ be fixed.
Since
	\begin{align*}
		&|Df(\rho(y, \phi, r, t))\rho(\eta, \chi, r, t) - Df(\rho(y, \phi, r_0, t))\rho(\eta, \chi, r_0, t)| \\
		&\le
		\|Df(\rho(y, \phi, r, t)) - Df(\rho(y, \phi, r_0, t))\| \|\rho(\eta, \chi, r, t)\| \\
			&\mspace{40mu} + \|Df(\rho(y, \phi, r_0, t))\| \|\rho(\eta, \chi, r, t) - \rho(\eta, \chi, r_0, t)\| \\
		&\le
		2\sup_{t \in [0, T]} \|Df(\rho(y, \phi, r, t)) - Df(\rho(y, \phi, r_0, t))\| \cdot \|(\eta, \chi)\| \\
			&\mspace{40mu} + \sup_{t \in [0, T]} \|Df(\rho(y, \phi, r_0, t))\|
			\cdot |(\eta + \bar{\chi})(t - r) - (\eta + \bar{\chi})(t - r_0)|
	\end{align*}
for all $t \in [0, T]$, $r \in [0, R]$, and $(\eta, \chi) \in \mathcal{W}^{1, p}[-R, T] \times \mathcal{W}^{1, p}[-R, 0]$, we have
	\begin{align*}
		&\|(A_{y, \phi, r} - A_{y, \phi, r_0})(\eta, \chi)\|_{\mathcal{W}^{1, p}[-R, T]} \\
		&=
		\biggl( \int_0^T
			|Df(\rho(y, \phi, r, t))\rho(\eta, \chi, r, t) - Df(\rho(y, \phi, r_0, t))\rho(\eta, \chi, r_0, t)|^p
		\mspace{2mu} \mathrm{d}t \biggr)^{\frac{1}{p}} \\
		&\le
		\biggl( 2T^{\frac{1}{p}} \sup_{t \in [0, T]} \|Df(\rho(y, \phi, r, t)) - Df(\rho(y, \phi, r_0, t))\| \biggr) \|(\eta, \chi)\| \\
			&\mspace{40mu} + \sup_{t \in [0, T]} \|Df(\rho(y, \phi, r_0, t))\|
			\biggl( \int_0^T
				|(\eta + \bar{\chi})(t - r) - (\eta + \bar{\chi})(t - r_0)|^p
			\mspace{2mu} \mathrm{d}t \biggr)^{\frac{1}{p}}
	\end{align*}
by the Minkowski inequality.
Here the estimate
	\begin{align*}
		\biggl( \int_0^T
			|(\eta + \bar{\chi})(t - r) - (\eta + \bar{\chi})(t - r_0)|^p
		\mspace{2mu} \mathrm{d}t \biggr)^{\frac{1}{p}}
		&=
		\biggl( \int_0^T
			\biggl| \int_{-r_0}^{-r} (\eta + \bar{\chi})'(t + \theta) \mspace{2mu} \mathrm{d}\theta \biggr|^p
		\mspace{2mu} \mathrm{d}t \biggr)^{\frac{1}{p}} \\
		&\le \|\eta + \bar{\chi}\|_{\mathcal{W}^{1, p}[-R, T]} |r - r_0| \\
		&\le |r - r_0| \cdot \|(\eta, \chi)\|
	\end{align*}
follows from from Corollary~\ref{cor:double integral, small order, L^p}.
As the conclusion, we obtain
	\begin{align*}
		\|A_{y, \phi, r} - A_{y, \phi, r_0}\|
		&\le 2T^{\frac{1}{p}} \sup_{t \in [0, T]} \|Df(\rho(y, \phi, r, t)) - Df(\rho(y, \phi, r_0, t))\| \\
		&\mspace{40mu} + \sup_{t \in [0, T]} \|Df(\rho(y, \phi, r_0, t))\| |r - r_0|,
	\end{align*}
which shows $\lim_{r \to r_0} \|A_{y, \phi, r} - A_{y, \phi, r_0}\| = 0$.

\textbf{Step 2.}
From Theorem~\ref{thm:C^1-smoothness wrt delay}, the function
	\begin{equation*}
		r \mapsto B_{y, \phi, r}
	\end{equation*}
is continuous for each fixed $(y, \phi) \in \mathcal{W}^{1, p}[-R, T] \times \mathcal{W}^{1, p}[-R, 0]$.
Therefore, we only have to show the pointwise equicontinuity of the family of functions
	\begin{equation*}
		(y, \phi) \mapsto B_{y, \phi, r},
	\end{equation*}
where $r \in [0, R]$.
This is indeed true in view of the following calculation:
	\begin{align*}
		&\|B_{y, \phi, r} - B_{y_0, \phi_0, r}\|_{\mathcal{W}^{1, p}[-R, T]} \\
		&=
		\biggl( \int_0^T
			|D_2f(\rho(y, \phi, r, t))(y + \bar{\phi})'(t - r) - D_2f(\rho(y_0, \phi_0, r, t))(y_0 + \bar{\phi}_0)'(t - r)|^p
		\mspace{2mu} \mathrm{d}t\biggr)^{\frac{1}{p}} \\
		&\le
		\biggl( \sup_{t \in [0, T]} \|D_2f(\rho(y, \phi, r, t)) - D_2f(\rho(y_0, \phi_0, r, t))\| \biggr)
		\|y + \bar{\phi}\|_{\mathcal{W}^{1, p}[-R, T]} \\
			&\mspace{40mu} + \biggl( \sup_{t \in [0, T]} \|D_2f(\rho(y_0, \phi_0, r, t))\| \biggr)
			\|(y - y_0, \phi - \phi_0)\|.
	\end{align*}
The detail has been omitted because this is similar to the case of the special case discussed in the previous subsection.

This completes the proof.
\end{proof}

\subsubsection{$C^1$-smooth dependence of solutions on initial histories and delay}

Let $1 \le p < \infty$ be given.

\begin{theorem}\label{thm:C^1-smooth dependence on history and delay, single const delay}
Let $B \subset \mathcal{W}^{1, p}([-R, 0], \mathbb{R}^N)$ be an open subset which is bounded with respect to the supremum norm.
Suppose that $f$ is of class $C^1$.
Then there exists $T > 0$ such that the function
	\begin{equation*}
		B \times [0, R] \ni (\phi, r) \mapsto x(\cdot; \phi, r)|_{[-R, T]} \in \mathcal{W}^{1, p}([-R, T], \mathbb{R}^N)
	\end{equation*}
is well-defined and continuously Fr\'{e}chet differentiable.
\end{theorem}

\begin{proof}
\textbf{Step 1.}
From the unique existence theorem (Proposition~\ref{prop:unique existence, single const delay}),
there is $T > 0$ such that the family of functions
	\begin{equation*}
		x(\cdot; \phi, r)|_{[-R, T]} \colon [-R, T] \to \mathbb{R}^N,
	\end{equation*}
where $(\phi, r) \in B \times [0, R]$, is well-defined, i.e.,
	\begin{equation*}
		T_{\phi, r} > T \mspace{20mu} (\forall (\phi, r) \in B \times [0, R]).
	\end{equation*}

\textbf{Step 2.}
By choosing small $T > 0$, we may assume that the family of maps
	\begin{equation*}
		\mathcal{T}(\cdot; \phi, r) \colon \bar{\varGamma}_{1, p}(\delta) \to \varGamma_{1, p}(\delta),
	\end{equation*}
where $(\phi, r) \in B \times [0, R]$, is a uniform contraction
by the uniform contraction lemma (Lemma~\ref{lem:uniform contraction, single const delay}).
Then the Banach fixed point theorem implies that $\mathcal{T}(\cdot; \phi, r)$ has the unique fixed point
$y(\cdot; \phi, r) \in \varGamma_{1, p}(\delta)$ for each $(\phi, r) \in B \times [0, R]$
because $\bar{\varGamma}_{1, p}(\delta)$ is a complete metric space.
By the uniqueness,
	\begin{equation*}
		x(\cdot; \phi, r)|_{[-R, T]} = y(\cdot; \phi, r) + \bar{\phi}_{[-R, T]}
	\end{equation*}
holds.

\textbf{Step 3.}
By the $C^1$-smoothness theorem
(Theorem~\ref{thm:C^1-smoothness wrt prolongation, history, and delay, single const delay}), the function
	\begin{equation*}
		\mathcal{T} \colon \varGamma_{1, p}(\delta) \times B \times [0, R] \to \varGamma_{1, p}(\delta)
	\end{equation*}
is continuously Fr\'{e}chet differentiable.
Therefore, Theorem~\ref{thm:C^1-uniform contraction thm}
together with Corollary~\ref{cor:uniqueness of a linear approximation} implies that
	\begin{equation*}
		B \times [0, R] \ni (\phi, r) \mapsto y(\cdot; \phi, r) \in \varGamma_{1, p}(\delta)
	\end{equation*}
is continuously Fr\'{e}chet differentiable.
This shows that
	\begin{equation*}
		B \times [0, R] \ni (\phi, r) \mapsto x(\cdot; \phi, r)|_{[-R, T]} \in \mathcal{W}^{1, p}[-R, T]
	\end{equation*}
is also continuously Fr\'{e}chet differentiable because
	\begin{equation*}
		\mathcal{W}^{1, p}[-R, 0] \ni \phi \mapsto \bar{\phi}|_{[-R, T]} \in \mathcal{W}^{1, p}[-R, T]
	\end{equation*}
is a continuous linear map.
\end{proof}

\subsubsection{$C^1$-smoothness of solution semiflow with a delay parameter}

Let $1 \le p < \infty$ be given.

\begin{theorem}\label{thm:C^1-smoothness of sol semiflow with delay, single const delay}
We define a map
	\begin{equation*}
		\varPhi \colon \mathbb{R}_+ \times \mathcal{W}^{1, p}([-R, 0], \mathbb{R}^N) \times [0, R] \supset \dom(\varPhi)
		\to \mathcal{W}^{1, p}([-R, 0], \mathbb{R}^N) \times [0, R]
	\end{equation*}
by
	\begin{equation*}
		\dom(\varPhi) = \bigcup_{(\phi, r) \in \mathcal{W}^{1, p}[-R, 0] \times [0, R]} [0, T_{\phi, r}) \times \{(\phi, r)\},
			\mspace{15mu}
		\varPhi(t, \phi, r) = (R_tx(\cdot; \phi, r), r).
	\end{equation*}
Suppose that $f$ is of class $C^1$.
Then $\varPhi$ is a $C^1$-maximal semiflow.
\end{theorem}

\begin{proof}
\textbf{Step 1.}
The unique existence theorem (Proposition~\ref{prop:unique existence, single const delay}) implies that
$\varPhi$ is a maximal semiflow with the escape time function $(\phi, r) \mapsto T_{\phi, r}$.

\textbf{Step 2.}
By the continuity of orbit (Lemma~\ref{lem:continuity of orbit, Sobolev history sp}),
	\begin{equation*}
		[0, T_{\phi, r}) \ni t \mapsto \varPhi(t, \phi, r) \in \mathcal{W}^{1, p}[-R, 0] \times [0, R]
	\end{equation*}
is continuous for every $(\phi, r) \in \mathcal{W}^{1, p}[-R, 0] \times [0, R]$.

\textbf{Step 3.}
Let $B \subset \mathcal{W}^{1, p}[-R, 0]$ be an open subset which is bounded with respect to the supremum norm.
By the $C^1$-smooth dependence theorem
(Theorem~\ref{thm:C^1-smooth dependence on history and delay, single const delay}), there is $T > 0$ such that
	\begin{equation*}
		B \times [0, R] \ni (\phi, r) \mapsto x(\cdot; \phi, r)|_{[-R, T]} \in \mathcal{W}^{1, p}[-R, T]
	\end{equation*}
is a well-defined continuously Fr\'{e}chet differentiable function, which implies
	\begin{equation*}
		[0, T] \times B \times [0, R] \subset \dom(\varPhi).
	\end{equation*}
By combining the above Fr\'{e}chet differentiability and the continuity of
	\begin{equation*}
		[0, T] \times \mathcal{W}^{1, p}[-R, T] \ni (t, x) \mapsto R_tx \in \mathcal{W}^{1, p}[-R, 0]
	\end{equation*}
(see Lemmas~\ref{lem:continuity of orbit, Sobolev history sp} and \ref{lem:continuity of history operators, Sobolev history sp}),
we obtain the following properties:
\begin{itemize}
\item the continuity of $\varPhi|_{[0, T] \times B \times [0, R]}$, i.e, 
	\begin{equation*}
		[0, T] \times B \times [0, R] \ni (t, \phi, r) \mapsto (R_tx(\cdot; \phi, r), r).
	\end{equation*}
\item the continuous Fr\'{e}chet differentiability of $\varPhi(t, \cdot, \cdot)|_{B \times [0, R]}$, i.e.,
	\begin{equation*}
		B \times [0, R] \ni (\phi, r) \mapsto (R_tx(\cdot; \phi, r), r)
	\end{equation*}
for each $t \in [0, T]$.
\end{itemize}

The above steps imply that $\varPhi$ is a $C^1$-maximal semiflow
from Theorems~\ref{thm:C^0-maximal semiflow} and \ref{thm:C^1-maximal semiflow}.
\end{proof}

\section{Comments and discussion}\label{sec:comments and discussion}

This paper reveals that the history spaces of Sobolev type $\mathcal{W}^{1, p}([-R, 0], \mathbb{R}^N)$ ($1 \le p < \infty$) arise
as the history spaces for the $C^1$-smooth dependence on initial histories and delay,
whose adoption is natural from the viewpoint of the differentiability of translation in $L^p$.
This paper also extends the regularity of initial histories from the Lipschitz continuity
and show that the topology induced by $\mathcal{W}^{1, p}$-norm is adapted,
where the history space of the Lipschitz continuous functions with the topology induced by $\mathcal{W}^{1, 1}$-norm
is used in the previous studies (see \cite{Hale--Ladeira 1991} and \cite{Hartung 2016b}).
Another feature of this paper is to prove the differentiability of solutions with respect to $r$ at $r = 0$.
It seems that there is some relationship with the $C^1$-smoothness of \textit{special flow} for the small delay
studied by Chicone~\cite{Chicone 2003}.

The extension of this work to the time- and state- dependent delay case will be a next task.
By a preparatory study,
it is expected that this extension explains a meaning of the strict monotonicity of the delayed argument function,
which is called the \textit{temporal order of reactions} by Walther~\cite{Walther 2009b}.
The study of the higher-order smoothness of solutions with respect to delay will also be a next task.
The results in Subsection~\ref{subsec:a special case} suggest that
it is appropriate to choose history spaces of higher-order Sobolev type,
where other spaces based on $W^{k, \infty}$ are used in previous studies
(see \cite{Chen--Hu--Wu 2010} and \cite{Hartung 2013d}).

\section*{Acknowledgment}
\addcontentsline{toc}{section}{Acknowledgment}

This work was supported by the Research Alliance Center for Mathematical Sciences, Tohoku University,
the Research Institute for Mathematical Sciences, an International Joint Usage/Research Center located in Kyoto University,
JSPS A3 Foresight Program, and JSPS KAKENHI Grant Number JP17H06460, JP19K14565.

\appendix
\section{Differentiability of translation in $L^p$}\label{sec:differentiability of translation in L^p}

We refer the reader to \cite{Tao 2011} and \cite{Brezis 2011} for general references of
theories of Lebesgue integration and Sobolev spaces, respectively.

\begin{lemma}\label{lem:double integral, small order, L^1}
Let $f \in L^1(\mathbb{R}, \mathbb{R})$ and $a < b$ be given real numbers.
Then for all $s, t \in \mathbb{R}$,
	\begin{equation*}
		\int_a^b \biggl| \int_s^t |f(x + y)| \mspace{2mu} \mathrm{d}y \biggr| \mspace{2mu} \mathrm{d}x
		\le \|f\|_{L^1(\mathbb{R})} |t - s|
	\end{equation*}
holds.
\end{lemma}

\begin{proof}
It is sufficient to consider the case $s < t$.
Let $A(t, s) \subset \mathbb{R}^2$ be the closed subset given by
	\begin{equation*}
		A(t, s) := \{\mspace{2mu} (x, y) : \text{$a \le x \le b$, $x + s \le y \le x + t$} \mspace{2mu}\},
	\end{equation*}
which is Lebesgue measurable.
Then the function
	\begin{equation*}
		\mathbb{R}^2 \ni (x, y) \mapsto |f(y)| 1_{A(t, s)}(x, y) \in \mathbb{R}
	\end{equation*}
is Lebesgue measurable.
Since for each fixed $x \in [a, b]$,
	\begin{equation*}
		\{\mspace{2mu} y \in \mathbb{R} : (x, y) \in A(t, s) \mspace{2mu}\} = [x + s, x + t],
	\end{equation*}
we have
	\begin{align*}
		\int_a^b \biggl( \int_s^t |f(x + y)| \mspace{2mu} \mathrm{d}y \biggr) \mspace{2mu} \mathrm{d}x
		&= \int_a^b \biggl( \int_{x + s}^{x + t} |f(y)| \mspace{2mu} \mathrm{d}y \biggr) \mspace{2mu} \mathrm{d}x \\
		&= \int_{[a, b]} \biggl( \int_\mathbb{R} |f(y)| 1_{A(t, s)}(x, y) \mspace{2mu} \mathrm{d}y \biggr) \mspace{2mu} \mathrm{d}x.
	\end{align*}
Therefore,
	\begin{align*}
		\int_a^b \biggl( \int_s^t |f(x + y)| \mspace{2mu} \mathrm{d}y \biggr) \mspace{2mu} \mathrm{d}x
		&= \int_\mathbb{R} \biggl( \int_{[a, b]} |f(y)| 1_{A(t, s)}(x, y) \mspace{2mu} \mathrm{d}x \biggr) \mspace{2mu} \mathrm{d}y \\
		&= \int_\mathbb{R} |f(y)| \biggl( \int_{[a, b]} 1_{A(t, s)}(x, y) \mspace{2mu} \mathrm{d}x \biggr) \mspace{2mu} \mathrm{d}y \\
		&\le (t - s) \|f\|_{L^1(\mathbb{R})} \\
		&< \infty
	\end{align*}
is valid by Tonelli's theorem.
\end{proof}

\begin{corollary}\label{cor:double integral, small order, L^p}
Let $1\le p < \infty$, $f \in L^p(\mathbb{R}, \mathbb{R}^N)$, and $a < b$ be given real numbers.
Then for all $s, t \in \mathbb{R}$,
	\begin{equation*}
		\biggl( \int_a^b
			\biggl| \int_s^t |f(x + y)| \mspace{2mu} \mathrm{d}y \biggr|^p
		\mspace{2mu} \mathrm{d}x \biggr)^{\frac{1}{p}}
		\le \|f\|_{L^p(\mathbb{R})} |t - s|
	\end{equation*}
holds.
\end{corollary}

\begin{proof}
It is sufficient to consider the case $s < t$.
Let $q$ be the H\"{o}lder conjugate of $p$.
Then for each fixed $x$, we have
	\begin{equation*}
		\int_s^t |f(x + y)| \mspace{2mu} \mathrm{d}y
		\le \biggl( \int_s^t |f(x + y)|^p \mspace{2mu} \mathrm{d}y \biggr)^{\frac{1}{p}} \cdot (t - s)^{\frac{1}{q}}.
	\end{equation*}
Since $|f|^p \in L^1(\mathbb{R}, \mathbb{R})$, we obtain
	\begin{align*}
		\int_a^b
			\biggl( \int_s^t |f(x + y)| \mspace{2mu} \mathrm{d}y \biggr)^p
		\mspace{2mu} \mathrm{d}x
		&\le (t - s)^{\frac{p}{q}}
			\int_a^b \biggl( \int_s^t |f(x + y)|^p \mspace{2mu} \mathrm{d}y \biggr) \mspace{2mu} \mathrm{d}x \\
		&\le (t - s)^{\frac{p}{q}} \cdot \||f|^p\|_{L^1(\mathbb{R})} (t - s) \\
		&\le (t - s)^{\frac{p}{q} + 1} \|f\|_{L^p(\mathbb{R})}^p
	\end{align*}
by applying Lemma~\ref{lem:double integral, small order, L^1}.
Then the inequality is obtained because $(1/p) + (1/q) = 1$.
\end{proof}

\begin{theorem}\label{thm:differentiability of translation in L^p}
Let $1 \le p < \infty$, $f \in L^p(\mathbb{R}, \mathbb{R}^N)$, and $a < b$ be real numbers.
Then for all $s, t, u \in \mathbb{R}$,
	\begin{equation*}
		\biggl( \int_a^b
			\biggl| \int_s^t (f(x + y) - f(x + u)) \mspace{2mu} \mathrm{d}y \biggr|^p
		\mspace{2mu} \mathrm{d}x \biggr)^{\frac{1}{p}}
		= o(|t - s|)
	\end{equation*}
as $|t - s| \to 0$ uniformly in $u$ between $s$ and $t$.
\end{theorem}

\begin{proof}
Let
	\begin{equation*}
		F(x; s, t, u) := \int_s^t (f(x + y) - f(x + u)) \mspace{2mu} \mathrm{d}y.
	\end{equation*}
Then for each fixed $x \in [a, b]$, we have
	\begin{align*}
		F(x; s, t, u)
		&= \int_s^t f(x + y) \mspace{2mu} \mathrm{d}y - (t - s)f(x + u) \\
		&= \int_{s + x}^{t + x} f(y) \mspace{2mu} \mathrm{d}y - (t - s)f(x + u),
	\end{align*}
which is Lebesgue measurable in $x$.

Let $\ep > 0$ be given.
We choose $g \in C_\mathrm{c}(\mathbb{R}, \mathbb{R}^N)$ so that
	\begin{equation*}
		\|f - g\|_{L^p(\mathbb{R})} \le \frac{\ep}{3}.
	\end{equation*}
Here $C_\mathrm{c}(\mathbb{R}, \mathbb{R}^N)$ denotes the set of continuous functions
from $\mathbb{R}$ to $\mathbb{R}^N$ with compact support.
By the Minkowski inequality, 
	\begin{align*}
		&\|F(\cdot; s, t, u)\|_{L^p[a, b]} \\
		&\le
		\biggl( \int_a^b
			\biggl| \int_s^t (f(x + y) - g(x + y)) \mspace{2mu} \mathrm{d}y \biggr|^p
		\mspace{2mu} \mathrm{d}x \biggr)^{\frac{1}{p}} \\
		&\mspace{40mu} +
		\biggl( \int_a^b
			\biggl| \int_s^t (g(x + y) - g(x + u)) \mspace{2mu} \mathrm{d}y \biggr|^p
		\mspace{2mu} \mathrm{d}x \biggr)^{\frac{1}{p}} \\
		&\mspace{80mu} +
		\biggl( \int_a^b
			\biggl| \int_s^t (g(x + u) - f(x + u)) \mspace{2mu} \mathrm{d}y \biggr|^p
		\mspace{2mu} \mathrm{d}x \biggr)^{\frac{1}{p}}.
	\end{align*}

\textbf{First term.} By applying Corollary~\ref{cor:double integral, small order, L^p}, we obtain
	\begin{align*}
		&\biggl( \int_a^b
			\biggl| \int_s^t (f(x + y) - g(x + y)) \mspace{2mu} \mathrm{d}y \biggr|^p
		\mspace{2mu} \mathrm{d}x \biggr)^{\frac{1}{p}} \\
		&\le
		\biggl( \int_a^b
			\biggl| \int_s^t |f(x + y) - g(x + y)| \mspace{2mu} \mathrm{d}y \biggr|^p
		\mspace{2mu} \mathrm{d}x \biggr)^{\frac{1}{p}} \\
		&\le \|f - g\|_{L^p(\mathbb{R})} |t - s|.
	\end{align*}

\textbf{Second term.} Since $g$ is uniformly continuous, there is $\delta > 0$ such that
for all $x, y, u$, $|y - u| \le \delta$ implies
	\begin{equation*}
		 |g(x + y) - g(x + u)| \le \frac{\ep}{3(b - a)^{1/p}}.
	\end{equation*}
Therefore, $|t - s| \le \delta$ implies
	\begin{equation*}
		\biggl| \int_s^t |g(x + y) - g(x + u)| \mspace{2mu} \mathrm{d}y \biggr|^p
		\le \left[ \frac{\ep}{3(b - a)^{1/p}} |t - s| \right]^p
	\end{equation*}
uniformly in $u$ between $s$ and $t$.
Thus,
	\begin{align*}
		&\biggl( \int_a^b
			\biggl| \int_s^t (g(x + y) - g(x + u)) \mspace{2mu} \mathrm{d}y \biggr|^p
		\mspace{2mu} \mathrm{d}x \biggr)^{\frac{1}{p}} \\
		&\le
		\biggl( \int_a^b
			\biggl| \int_s^t |g(x + y) - g(x + u)| \mspace{2mu} \mathrm{d}y \biggr|^p
		\mspace{2mu} \mathrm{d}x \biggr)^{\frac{1}{p}} \\
		&\le \left\{ \left[ \frac{\ep}{3(b - a)^{1/p}} |t - s| \right]^p (b - a) \right\}^{\frac{1}{p}} \\
		&\le \frac{\ep}{3} |t - s|.
	\end{align*}

\textbf{Third term.} We have
	\begin{align*}
		&\biggl( \int_a^b
			\biggl| \int_s^t (g(x + u) - f(x + u)) \mspace{2mu} \mathrm{d}y \biggr|^p
		\mspace{2mu} \mathrm{d}x \biggr)^{\frac{1}{p}} \\
		&\le
		\biggl( \int_a^b
			\biggl| \int_s^t |g(x + u) - f(x + u)| \mspace{2mu} \mathrm{d}y \biggr|^p
		\mspace{2mu} \mathrm{d}x \biggr)^{\frac{1}{p}} \\
		&= \biggl( |t - s|^p \int_a^b |g(x + u) - f(x + u)|^p \mspace{2mu} \mathrm{d}x \biggr)^{\frac{1}{p}} \\
		&\le \|f - g\|_{L^p(\mathbb{R})} |t - s|.
	\end{align*}

By combining the above estimates, we finally obtain
	\begin{equation*}
		\biggl( \int_a^b
			\biggl| \int_s^t (f(x + y) - f(x + u)) \mspace{2mu} \mathrm{d}y \biggr|^p
		\mspace{2mu} \mathrm{d}x \biggr)^{\frac{1}{p}}
		\le \ep |t - s|
	\end{equation*}
for all $|t - s| \le \delta$ uniformly in $u$ between $s$ and $t$.
\end{proof}

\begin{corollary}\label{cor:differentiability of translation in L^p}
Let $1 \le p < \infty$.
Let $a < b$ and $c, d \ge 0$ be given real numbers.
If $f \in \mathcal{W}^{1, p}([a - c, b + d], \mathbb{R}^N)$, then for all $s, t, u \in [-c, d]$,
	\begin{equation*}
		\biggl( \int_a^b
			| f(x + t) - f(x + s) - (t - s)f'(x + u) |^p
		\mspace{2mu} \mathrm{d}x \biggr)^{\frac{1}{p}}
		= o(|t - s|)
	\end{equation*}
as $|t - s| \to 0$ uniformly in $u$ between $s$ and $t$.
\end{corollary}

\begin{proof}
Let $x \in [a, b]$ and $s, t \in [-c, d]$.
By the fundamental theorem of calculus for absolutely continuous functions, we have
	\begin{equation*}
		f(x + t) - f(x + s) = \int_s^t f'(x + y) \mspace{2mu} \mathrm{d}y,
	\end{equation*}
which implies that for all $u$ between $s$ and $t$,
	\begin{equation*}
		f(x + t) - f(x + s) - (t - s)f'(x + u)
		= \int_s^t (f'(x + y) - f'(x + u)) \mspace{2mu} \mathrm{d}y.
	\end{equation*}
Therefore, the conclusion is obtained applying Theorem~\ref{thm:differentiability of translation in L^p}
for the extension of $f' \in L^p([a - c, b + d], \mathbb{R}^N)$ by $0$ outside $[a - c, b + d]$.
\end{proof}

\begin{remark}
The similar statement is given in \cite[Exercise 8.13 in Chapter 8]{Brezis 2011}.
\end{remark}

\section{$C^1$-uniform contraction theorem}\label{sec:C^1-uniform contraction theorem}

We first define the Fr\'{e}chet differentiability of functions whose domain of definitions are not necessarily open.

\begin{notation}
For normed spaces $X$ and $Y$,
the set of all bounded linear operators from $X$ to $Y$ is denoted by $\mathcal{L}(X, Y)$.
It is a Banach space with the operator norm defined by
	\begin{equation*}
		\|A\| := \sup\{\mspace{2mu} \|Ax\| : \text{$x \in X$, $\|x\| \le 1$} \mspace{2mu}\}
	\end{equation*}
for any $A \in \mathcal{L}(X, Y)$.
\end{notation}

\begin{definition}[Fr\'{e}chet differentiability]\label{dfn:Frechet differentiability}
Let $X, Y$ be normed spaces, $U \subset X$ be a subset, $x_0 \in U$ be a limit point of $U$,
and $f \colon U \to Y$ be a function.
$f$ is said to be \textit{Fr\'{e}chet differentiable} at $x_0$
if there exists a unique $A \in \mathcal{L}(X, Y)$ such that
	\begin{equation}\label{eq:Frechet differentiability}
		\lim_{\text{$\|x - x_0\| \to 0$ in $U$}} \frac{\|f(x) - f(x_0) - A(x - x_0)\|}{\|x - x_0\|} = 0.
	\end{equation}
The above $A$ is called the \textit{Fr\'{e}chet derivative} of $f$ at $x_0$ and is denoted by $Df(x_0)$.
$f$ is said to be \textit{Fr\'{e}chet differentiable}
when $U$ is contained in the set of all limit points of $U$ and $f$ is Fr\'{e}chet differentiable at every $x_0 \in U$.
\end{definition}

The above $A$ is a linear approximation of $f$ at $x_0$.
When a linear approximation of $f$ at $x_0$ is unique and continuous,
we say that $f$ is Fr\'{e}chet differentiable at $x_0$.

\begin{remark}
\eqref{eq:Frechet differentiability} is equivalent to the following condition:
For every $\ep > 0$, there exists $\delta > 0$ such that for all $x \in U$, $\|x - x_0\| \le \delta$ implies
	\begin{equation*}
		\|f(x) - f(x_0) - A(x - x_0)\| \le \ep\|x - x_0\|.
	\end{equation*}
We note that the both sides are equal to $0$ when $x = x_0$.
The above property is written as
	\begin{equation*}
		\|f(x) - f(x_0) - A(x - x_0)\| = o(\|x - x_0\|)
	\end{equation*}
as $\|x - x_0\| \to 0$ in $U$.
\end{remark}

\begin{definition}[Continuous Fr\'{e}chet differentiability]
Let $X, Y$ be normed spaces, $U \subset X$ be a subset contained in the set of all limit points of $U$,
and $f \colon U \to Y$ be a function.
$f$ is said to be \textit{continuously Fr\'{e}chet differentiable}
if $f$ is Fr\'{e}chet differentiable and $Df \colon U \to \mathcal{L}(X, Y)$ is a continuous map.
\end{definition}

The following is a basic result about the uniqueness of a linear approximation.
We omit the proof.

\begin{lemma}\label{lem:uniqueness of a linear approximation}
Let $X, Y$ be normed spaces, $U \subset X$ be an open subset, $x_0 \in U$, and $f \colon U \to Y$ be a function.
Suppose that there exist linear maps $A, B \colon X \to Y$ such that
	\begin{align*}
		\|f(x) - f(x_0) - A(x - x_0)\| &= o(\|x - x_0\|), \\
		\|f(x) - f(x_0) - B(x - x_0)\| &= o(\|x - x_0\|)
	\end{align*}
as $\|x - x_0\| \to 0$ in $U$.
Then $A = B$.
\end{lemma}

\begin{theorem}\label{thm:uniqueness of Frechet derivative, basis}
Let $N \ge 1$ be an integer, $Y$ be a normed space, $U \subset \R^N$ be a subset,
$x_0 \in U$ be a limit point of $U$, and $f \colon U \to Y$ be a function.
Suppose that there exist linear maps $A, B \colon \R^N \to Y$ such that
	\begin{align*}
		\|f(x) - f(x_0) - A(x - x_0)\| &= o(|x - x_0|), \\
		\|f(x) - f(x_0) - B(x - x_0)\| &= o(|x - x_0|)
	\end{align*}
as $|x - x_0| \to 0$ in $U$.
If there exists a basis $(v_1, \dots, v_N)$ of $\R^N$ such that
	\begin{equation*}
		\{\mspace{2mu} x_0 + a_1v_1 + \dots + a_Nv_N : a_i \in [0, 1] \mspace{2mu}\} \subset U,
	\end{equation*}
then $A = B$.
\end{theorem}

\begin{proof}
Let $\ep > 0$ be given.
The assumption implies that there is $\delta > 0$ such that for all $x \in U$, $|x - x_0| \le \delta$ implies
	\begin{align*}
		\|f(x) - f(x_0) - A(x - x_0)\| &\le \frac{\ep}{2} |x - x_0|, \\
		\|f(x) - f(x_0) - B(x - x_0)\| &\le \frac{\ep}{2} |x - x_0|,
	\end{align*}
both of which show
	\begin{align*}
		\|(A - B)(x - x_0)\|
		&\le \|A(x - x_0) - (f(x) - f(x_0))\| + \|(f(x) - f(x_0)) - B(x - x_0)\| \\
		&\le \frac{\ep}{2} |x - x_0| + \frac{\ep}{2} |x - x_0| \\
		&= \ep|x - x_0|.
	\end{align*}
Fix $i \in \{1, \dots, N\}$.
Then for all $h \in \R$, $0 < |h| \le \min\{\delta/|v_i|, 1\}$ implies
	\begin{equation*}
		\|(A - B)hv_i\| \le \ep |h||v_i|.
	\end{equation*}
This means $\|(A - B)v_i\| \le \ep|v_i|$, and we have $(A - B)v_i = 0$ since $\ep > 0$ is arbitrary.
This holds for each fixed $i$, and therefore,
	\begin{equation*}
		A = B
	\end{equation*}
because $(v_1, \dots, v_N)$ is a basis of $\R^N$.
\end{proof}

\begin{corollary}\label{cor:uniqueness of a linear approximation}
Let $X, Y$ be normed spaces, $N \ge 1$ be an integer, $U \subset X$ be an open subset,
$V \subset \R^N$ be a subset, $x_0 \in U$, $y_0 \in V$ be a limit point of $V$,
and $f \colon U \times V \to Y$ be a function.
Suppose that there exist linear maps $A, B  \colon X \times \R^N \to Y$ such that
	\begin{align*}
		\|f(x, y) - f(x_0, y_0) - A(x - x_0, y - y_0)\| &= o(\|x - x_0\| + |y - y_0|), \\
		\|f(x, y) - f(x_0, y_0) - B(x - x_0, y - y_0)\| &= o(\|x - x_0\| + |y - y_0|)
	\end{align*}
as $\|x - x_0\| + |y - y_0| \to 0$ in $U \times V$.
If there exists a basis $(v_1, \dots, v_N)$ of $\R^N$ such that
	\begin{equation*}
		\{\mspace{2mu} y_0 + a_1v_1 + \dots + a_Nv_N : a_i \in [0, 1] \mspace{2mu}\} \subset V,
	\end{equation*}
then $A = B$.
\end{corollary}

\begin{proof}
By combining Lemma~\ref{lem:uniqueness of a linear approximation}
and Theorem~\ref{thm:uniqueness of Frechet derivative, basis}, we have
	\begin{equation*}
		A(u, 0) = B(u, 0) \mspace{10mu} \text{and} \mspace{10mu} A(0, v) = B(0, v)
	\end{equation*}
for all $u \in X$ and all $v \in \R^N$.
By the linearity,
	\begin{equation*}
		A(u, v)
		= A(u, 0) + A(0, v)
		= B(u, 0) + B(0, v)
		= B(u, v)
	\end{equation*}
holds for all $(u, v) \in X \times \R^N$, which means $A = B$.
\end{proof}

\begin{lemma}\label{lem:linear approximation and continuity}
Let $X, Y$ be normed spaces, $U \subset X$ be a subset, $x_0 \in U$ be a limit point of $U$,
and $f \colon U \to Y$ be a function.
If there exists $A \in \mathcal{L}(X, Y)$ such that
	\begin{equation*}
		\|f(x) - f(x_0) - A(x - x_0)\| = o(\|x - x_0\|)
	\end{equation*}
as $\|x - x_0\| \to 0$ in $U$, then there exist positive numbers $\delta, L$ such that
for all $x \in U$, $\|x - x_0\| \le \delta$ implies
	\begin{equation*}
		\|f(x) - f(x_0)\| \le L\|x - x_0\|.
	\end{equation*}
\end{lemma}

\begin{remark}
The conclusion does not mean the local Lipschitz continuity of $f$ at $x_0$
but means the continuity of $f$ at $x_0$. 
\end{remark}

\begin{proof}
Let $\ep > 0$.
By the assumption, there exists $\delta > 0$ such that for all $x \in U$, $\|x - x_0\| \le \delta$ implies
	\begin{equation*}
		\|f(x) - f(x_0) - A(x - x_0)\| \le \|x - x_0\|.
	\end{equation*}
Therefore, we have
	\begin{align*}
		\|f(x) - f(x_0)\|
		&\le \|f(x) - f(x_0) - A(x - x_0)\| + \|A(x - x_0)\| \\
		&\le (1 + \|A\|)\|x - x_0\|
	\end{align*}
for all $x \in U$ satisfying $\|x - x_0\| \le \delta$.
\end{proof}

\begin{lemma}\label{lem:Frechet differentiability and Lipschitz continuity}
Let $X, Y$ be normed spaces, $U \subset X$ be an open subset, $x_0 \in U$, and $f \colon U \to Y$ be a function.
Suppose that there exists a linear map $A \colon X \to Y$ such that
	\begin{equation*}
		\|f(x) - f(x_0) - A(x - x_0)\| = o(\|x - x_0\|)
	\end{equation*}
as $\|x - x_0\| \to 0$.
If there exists $L > 0$ such that
	\begin{equation*}
		\|f(x_1) - f(x_2)\| \le L\|x_1 - x_2\|
	\end{equation*}
for all $x_1, x_2 \in U$, then $A \in \mathcal{L}(X, Y)$ and $\|A\| \le L$.
\end{lemma}

\begin{proof}
Let $\ep > 0$ be given.
By the assumption, there is $\delta > 0$ such that for all $v \in X$, $\|v\| \le \delta$ implies
	\begin{equation*}
		x_0 + v \in U
			\mspace{10mu} \text{and} \mspace{10mu}
		 \|f(x_0 + v) - f(x_0) - Av\| \le \ep\|v\|.
	\end{equation*}
Therefore,
	\begin{align*}
		\|Av\|
		&\le \|Av - (f(x_0 + v) - f(x_0))\| + \|f(x_0 + v) - f(x_0)\| \\
		&\le (\ep + L)\|v\|
	\end{align*}
for all $v \in X$ satisfying $\|v\| \le \delta$.
This shows $\|A\| \le L + \ep$.
Since $\ep > 0$ is arbitrary, $\|A\| \le L$ holds.
\end{proof}

The following is the $C^1$-uniform contraction theorem which is used in this paper.

\begin{theorem}\label{thm:C^1-uniform contraction thm}
Let $X$ be a Banach space, $\Lambda$ be a normed space, $U \subset X$ be an open subset,
$V \subset \Lambda$ be a subset contained in the set of all limit points of $V$,
and $T \colon U \times V \to X$ be a map.
Suppose that there exists a map $g \colon V \to U$ such that
$g(\lambda)$ is a fixed point of $T(\cdot, \lambda) \colon U \to X$ for all $\lambda \in V$
and the family of maps
	\begin{equation*}
		T(\cdot, \lambda) \colon U \to X,
	\end{equation*}
where $\lambda \in V$, is a uniform contraction.
If $T$ is Fr\'{e}chet differentiable, then for each $\lambda_0 \in V$,
	\begin{equation*}
		\|g(\lambda) - g(\lambda_0) - A_{\lambda_0}(\lambda - \lambda_0)\| = o(\|\lambda - \lambda_0\|)
	\end{equation*}
as $\|\lambda - \lambda_0\| \to 0$ in $V$.
Here the bounded linear operator $A_\lambda \colon \Lambda \to X$ is defined by
	\begin{equation*}
		A_\lambda := [1 - D_1T(g(\lambda), \lambda)]^{-1}D_2T(g(\lambda), \lambda)
	\end{equation*}
for every $\lambda \in V$.
\end{theorem}

\begin{proof}
We choose $c \in (0, 1)$ so that for all $(x_1, \lambda), (x_2, \lambda) \in U \times V$
	\begin{equation*}
		\|T(x_1, \lambda) - T(x_2, \lambda)\| \le c\|x_1 - x_2\|
	\end{equation*}
holds.
Let $\lambda_0 \in V$ be given.

\textbf{Step 1.}
By the assumption, $T(g(\lambda_0), \cdot) \colon V \to X$ is Fr\'{e}chet differentiable at $\lambda_0$.
From Lemma~\ref{lem:linear approximation and continuity},
we can choose positive numbers $\delta, L$ so that
for all $\lambda \in V$, $\|\lambda - \lambda_0\| \le \delta$ implies
	\begin{equation*}
		\|T(g(\lambda_0), \lambda) - T(g(\lambda_0), \lambda_0)\| \le L\|\lambda - \lambda_0\|.
	\end{equation*}
Since $g(\lambda)$ is the fixed point of $T(\cdot, \lambda)$ for each $\lambda \in V$,
	\begin{align*}
		\|g(\lambda) - g(\lambda_0)\|
		&= \|T(g(\lambda), \lambda) - T(g(\lambda_0), \lambda_0)\| \\
		&\le \|T(g(\lambda), \lambda) - T(g(\lambda_0), \lambda)\|
		+ \|T(g(\lambda_0), \lambda) - T(g(\lambda_0), \lambda_0)\| \\
		&\le c\|g(\lambda) - g(\lambda_0)\| + \|T(g(\lambda_0), \lambda) - T(g(\lambda_0), \lambda_0)\|.
	\end{align*}
By combining the above inequalities,
	\begin{align*}
		\|g(\lambda) - g(\lambda_0)\|
		&\le \frac{\|T(g(\lambda_0), \lambda) - T(g(\lambda_0), \lambda_0)\|}{1 - c} \\
		&\le \frac{L}{1 - c}\|\lambda - \lambda_0\|
	\end{align*}
holds for all $\lambda \in V$ satisfying $\|\lambda - \lambda_0\| \le \delta$.

\textbf{Step 2.}
By the assumption, the map $T(\cdot, \lambda) \colon U \to X$ is Lipschitz continuous with a Lipschitz constant $c$
and is Fr\'{e}chet differentiable for all $\lambda \in V$.
Applying Lemma~\ref{lem:Frechet differentiability and Lipschitz continuity},
	\begin{equation*}
		\|D_1T(x, \lambda)\| \le c
	\end{equation*}
holds for every $(x, \lambda) \in U \times V$.
In particular, $D_1T(g(\lambda_0), \lambda_0) \in \mathcal{L}(X)$ has the bounded inverse
because $X$ is a Banach space.
By using
	\begin{equation*}
		A_{\lambda_0} = D_1T(g(\lambda_0), \lambda_0)A_{\lambda_0} + D_2T(g(\lambda_0), \lambda_0),
	\end{equation*}
we have 
	\begin{align*}
		&g(\lambda) - g(\lambda_0) - A_{\lambda_0}(\lambda - \lambda_0) \\
		&= T(g(\lambda), \lambda) - T(g(\lambda_0), \lambda_0)
			- [D_1T(g(\lambda_0), \lambda_0)A_{\lambda_0}(\lambda - \lambda_0)
			+ D_2T(g(\lambda_0), \lambda_0)(\lambda - \lambda_0)]
	\end{align*}
for all $\lambda \in V$.
This implies
	\begin{align*}
		&[1 - D_1T(g(\lambda_0), \lambda_0)][g(\lambda) - g(\lambda_0) - A_{\lambda_0}(\lambda - \lambda_0)] \\
		&= T(g(\lambda), \lambda) - T(g(\lambda_0), \lambda_0)
			- DT(g(\lambda_0), \lambda_0)(g(\lambda) - g(\lambda_0), \lambda - \lambda_0)
	\end{align*}
by a simple calculation.

\textbf{Step 3.}
Let $\ep > 0$ be given.
Since $T$ is Fr\'{e}chet differentiable at $(g(\lambda_0), \lambda_0)$,
there is $\delta' > 0$ such that
for all $(x, \lambda) \in U \times V$, $\|x - g(\lambda_0)\| + \|\lambda - \lambda_0\| \le \delta'$ implies
	\begin{align*}
		&\|T(x, \lambda) - T(g(\lambda_0), \lambda_0)
		- DT(g(\lambda_0), \lambda_0)(x - g(\lambda_0), \lambda - \lambda_0)\| \\
		&\le \ep(\|x - g(\lambda_0)\| + \|\lambda - \lambda_0\|).
	\end{align*}
By the continuity of $g$ at $\lambda_0$,
we can choose $\delta'' \in (0, \delta)$ such that
for all $\lambda \in V$, $\|\lambda - \lambda_0\| \le \delta''$ implies
	\begin{equation*}
		\|g(\lambda) - g(\lambda_0)\| + \|\lambda - \lambda_0\| \le \delta',
	\end{equation*}
and therefore,
	\begin{align*}
		&\|T(g(\lambda), \lambda) - T(g(\lambda_0), \lambda_0)
			- DT(g(\lambda_0), \lambda_0)(g(\lambda) - g(\lambda_0), \lambda - \lambda_0)\| \\
		&\le \ep(\|g(\lambda) - g(\lambda_0)\| + \|\lambda - \lambda_0\|) \\
		&\le \left( \frac{L}{1 - c} + 1 \right) \ep \|\lambda - \lambda_0\|.
	\end{align*}
By combining Step 2 with this inequality, we obtain
	\begin{equation*}
		\|g(\lambda) - g(\lambda_0) - A_{\lambda_0}(\lambda - \lambda_0)\|
		= o(\|\lambda - \lambda_0\|)
	\end{equation*}
as $\|\lambda - \lambda_0\| \to 0$ in $V$.
\end{proof}

\begin{remark}
If $T$ is continuously Fr\'{e}chet differentiable,
$\lambda \mapsto A_\lambda$ is continuous.
\end{remark}

\section{Continuity and smoothness of maximal semiflows}\label{sec:maximal semiflows}

\begin{definition}[Maximal semiflows]\label{dfn:maximal semiflows}
Let $X$ be a set and $D \subset \R_+ \times X$ be a subset.
A map $\varPhi \colon D \to X$ is called a \textit{maximal semiflow} in $X$ if the following conditions are satisfied:
\begin{enumerate}
\item[(i)] There exists a function $T_\varPhi \colon X \to (0, \infty]$ such that
	\begin{equation*}
		D = \bigcup_{x \in X} \bigl( [0, T_\varPhi(x)) \times \{x\} \bigr).
	\end{equation*}
\item[(ii)] For all $x \in X$, $\varPhi(0, x) = x$.
\item[(iii)] For all $t, s \in \R_+$ and all $x \in X$,
both of the conditions $(t, x) \in D$ and $(s, \varPhi(t, x)) \in D$ imply 
	\begin{equation*}
		(t + s, x) \in D
			\mspace{10mu} \text{and} \mspace{10mu}
		\varPhi(t + s, x) = \varPhi(s, \varPhi(t, x)).
	\end{equation*}
\end{enumerate}
The above function $T_\varPhi$ is called the \textit{escape time function}.
\end{definition}

\begin{remark}
The condition (iii) means the maximality of domain of definition of $\varPhi$.
In terms of the escape time function $T_\varPhi$, (iii) is equivalent to the following:
both of $t < T_\varPhi(x)$ and $s < T_\varPhi(\varPhi(t, x))$ imply $t + s < T_\varPhi(x)$.
The terminology of maximal semiflows comes from \cite{Marsden--McCracken 1976}.
\end{remark}

\begin{definition}[Time-$t$ map]
Let $\varPhi$ be a maximal semiflow in a set $X$
with the escape time function $T_\varPhi \colon X \to (0, \infty]$.
For each $t \in \R_+$, the map $\varPhi^t \colon \dom(\varPhi^t) \to X$ defined by
	\begin{equation*}
		\dom(\varPhi^t) = \{\mspace{2mu} x \in X : T_\varPhi(x) > t \mspace{2mu}\}
			\mspace{10mu} \text{and} \mspace{10mu}
		\varPhi^t(x) = \varPhi(t, x)
	\end{equation*}
is called the \textit{time-$t$ map} of $\varPhi$.
\end{definition}

\begin{definition}[Lower semicontinuity]
Let $X$ be a topological space, $x_0 \in X$, and $f \colon X \to (0, \infty]$ be a function.
$f$ is said to be \textit{lower semicontinuous} at $x_0$
if for every $M < f(x_0)$, there exists a neighborhood $N$ of $x_0$ such that 
for all $x \in N$, $f(x) > M$.
$f$ is said to be \textit{lower semicontinuous} if $f$ is lower semicontinuous at every $x_0 \in X$.
\end{definition}

\begin{definition}[$C^0$-maximal semiflows]\label{dfn:C^0-maximal semiflows}
Let $X$ be a topological space and $\varPhi \colon \dom(\varPhi) \to X$ be a maximal semiflow in $X$.
$\varPhi$ is called a \textit{$C^0$-maximal semiflow}
if $\varPhi$ is a continuous map and the escape time function $T_\varPhi \colon X \to (0, \infty]$ is lower semicontinuous.
\end{definition}

\begin{remark}
In \cite{Hajek 1968}, a $C^0$-maximal semiflow is called a continuous local semi-dynamical system.
\end{remark}

The proofs of the following two lemmas are straightforward and can be omitted.

\begin{lemma}
Let $\varPhi \colon \dom(\varPhi) \to X$ be a maximal semiflow in a topological space $X$
with the escape time function $T_\varPhi \colon X \to (0, \infty]$.
Then the following properties are equivalent:
\begin{enumerate}
\item[(a)] $T_\varPhi \colon X \to (0, \infty]$ is lower semicontinuous.
\item[(b)] $\dom(\varPhi)$ is an open set of $\R_+ \times X$.
\end{enumerate}
\end{lemma}

\begin{lemma}\label{lem:domain of time-t map}
Let $\varPhi$ be a $C^0$-maximal semiflow in a topological space $X$
with the escape time function $T_\varPhi \colon X \to (0, \infty]$.
Then for each $t \in \R_+$,
	\begin{equation*}
		\{\mspace{2mu} x \in X : T_\varPhi(x) > t \mspace{2mu}\}
	\end{equation*}
is an open subset of $X$.
\end{lemma}

The following theorem states that the local continuity property of maximal semiflows can induce
their global continuity property.
We omit the proof because a similar statement is proved in \cite[Theorem A.7]{Nishiguchi 2017}.

\begin{theorem}\label{thm:C^0-maximal semiflow}
Let $\varPhi \colon \dom(\varPhi) \to X$ be a maximal semiflow in a topological space $X$
with the escape time function $T_\varPhi \colon X \to (0, \infty]$.
Suppose that for every $x \in X$, the orbit $[0, T_\varPhi(x)) \ni t \mapsto \varPhi(t, x) \in X$ is continuous.
If for every $x \in X$, there exist $T > 0$ and a neighborhood $N$ of $x$ in such that
$[0, T] \times N \subset \dom(\varPhi)$ and $\varPhi|_{[0, T] \times N}$ is continuous,
then $\varPhi$ is a $C^0$-maximal semiflow.
\end{theorem}

\begin{remark}
In \cite[Theorem 15]{Hajek 1968}, the conclusion is obtained under the weaker assumption that
for every $(t, x) \in \dom(\varPhi)$, $\varPhi([0, t] \times \{x\})$ is compact.
The proof is based on the notion of germs.
\end{remark}

\begin{definition}[$C^1$-maximal semiflows]\label{dfn:C^1-maximal semiflows}
Let $X$ be a normed space and $\varOmega \subset X$ be a subset contained in the set of all limit points of $\varOmega$.
A $C^0$-maximal semiflow $\varPhi \colon \dom(\varPhi) \to \varOmega$ is called
a \textit{$C^1$-maximal semiflow}
if each time-$t$ map $\varPhi^t$ is continuously Fr\'{e}chet differentiable.
\end{definition}

\begin{remark}
In the setting of Definition~\ref{dfn:C^1-maximal semiflows},
$\dom(\varPhi^t)$ is open in $\varOmega$ from Lemma~\ref{lem:domain of time-t map}.
Therefore, $\dom(\varPhi^t) = U \cap \varOmega$ holds for some open set $U$ of $X$.
This implies that $\dom(\varPhi^t)$ is also contained in the set of all limit points of $\dom(\varPhi^t)$,
and it is meaningful to consider the continuous Fr\'{e}chet differentiability of each $\varPhi^t$.
\end{remark}

By definition, a $C^1$-maximal semiflow is not necessarily continuously Fr\'{e}chet differentiable
(see \cite[p. 260]{Marsden--McCracken 1976}).

The following theorem ensures that a $C^0$-maximal semiflow is of class $C^1$
provided that the maximal semiflow has a local smoothness property.
The proof is similar to that of \cite[Theorem 1]{Walther 2003c}.

\begin{theorem}\label{thm:C^1-maximal semiflow}
Let $X$ be a normed space, $\varOmega \subset X$ be a subset contained in the set of all limit points of $\varOmega$,
and $\varPhi \colon \dom(\varPhi) \to \varOmega$ be a $C^0$-maximal semiflow
with the escape time function $T_\varPhi \colon \varOmega \to (0, \infty]$.
Suppose that any function $f \colon \varOmega \to X$ has a unique linear approximation
at every $x \in \varOmega$.
If for every $x \in \varOmega$, there exist $T > 0$ and an open neighborhood $N$ of $x$ such that
\begin{itemize}
\item $[0, T] \times N \cap \varOmega \subset \dom(\varPhi)$ and
\item $\varPhi^t|_{N \cap \varOmega}$ is continuously Fr\'{e}chet differentiable for every $t \in [0, T]$,
\end{itemize}
then $\varPhi$ is a $C^1$-maximal semiflow.
\end{theorem}

\begin{proof}
\textbf{Step 1.}
For each $x \in \varOmega$, we define a subset $S_x \subset (0, T_\varPhi(x))$ by the following manner:
$T \in S_x$ if there exists an open neighborhood $N$ of $x$ such that
\begin{itemize}
\item $[0, T] \times N \cap \varOmega \subset \dom(\varPhi)$ and
\item $\varPhi^t|_{N \cap \varOmega}$ is continuously Fr\'{e}chet differentiable for every $t \in [0, T]$.
\end{itemize}
By the assumptions, $S_x \ne \emptyset$, and therefore, $\sup(S_x) \in (0, T_\varPhi(x)]$ exists.
If $\sup(S_x) = T_\varPhi(x)$ for all $x \in \varOmega$, then every $\varPhi^t$ is continuously Fr\'{e}chet differentiable.

Let $x_0 \in \varOmega$ be fixed.

\textbf{Step 2.}
We suppose
	\begin{equation*}
		t_* := \sup(S_{x_0}) < T_\varPhi(x_0)
	\end{equation*}
and derive a contradiction.
We note that one cannot conclude $t_* \in S_{x_0}$ in general.
Let
	\begin{equation*}
		x_* := \varPhi(t_*, x_0) \in \varOmega.
	\end{equation*}
By the assumptions, we can choose $T_* > 0$ and an open neighborhood $N_*$ of $x_*$ so that
\begin{itemize}
\item $[0, T_*] \times N_* \cap \varOmega \subset \dom(\varPhi)$ and
\item $\varPhi^t|_{N_* \cap \varOmega}$ is continuously Fr\'{e}chet differentiable for every $t \in [0, T_*]$.
\end{itemize}

\textbf{Step 3.}
Since $[0, T_\varPhi(x_0)) \ni t \mapsto \varPhi(t, x_0)$ is continuous at $t_*$, we can choose $t'$ so that
	\begin{equation*}
		t_* - \frac{T_*}{2} < t' < t_*
			\mspace{10mu} \text{and} \mspace{10mu}
		\varPhi(t', x_0) \in N_* \cap \varOmega.
	\end{equation*}
We can also choose an open neighborhood $N'$ of $x_0$ such that
\begin{itemize}
\item $[0, t'] \times N' \cap \varOmega \subset \dom(\varPhi)$,
\item $\varPhi^t|_{N' \cap \varOmega}$ is continuously Fr\'{e}chet differentiable for every $t \in [0, t']$, and
\item $\varPhi^{t'}(N' \cap \varOmega) \subset N_* \cap \varOmega$
\end{itemize}
because $t' < t_*$ and $\varPhi^{t'}$ is continuous at $x_0$.
Then for all $t \in [t', t' + T_*]$ and all $x \in N' \cap \varOmega$,
	\begin{equation*}
		(t', x) \in \dom(\varPhi)
			\mspace{10mu} \text{and} \mspace{10mu}
		(t - t', \varPhi(t', x)) \in [0, T_*] \times N_* \cap \varOmega \subset \dom(\varPhi),
	\end{equation*}
which implies
	\begin{equation*}
		(t, x) = (t' + (t - t'), x) \in \dom(\varPhi)
	\end{equation*}
by the maximality.
Therefore,
	\begin{align*}
		[0, t' + T_*] \times N' \cap \varOmega
		&= ([0, t'] \times N' \cap \varOmega) \cup ([t', t' + T_*] \times N' \cap \varOmega) \\
		&\subset \dom(\varPhi).
	\end{align*}

\textbf{Step 4.}
For every $t \in [t', t' + T_*]$, we have
	\begin{equation*}
		\varPhi^t|_{N' \cap \varOmega}
		= \varPhi^{t - t'}|_{N* \cap \varOmega} \circ \varPhi^{t'}|_{N' \cap \varOmega}.
	\end{equation*}
Since the two maps in the right-hand side are continuously Fr\'{e}chet differentiable,
$\varPhi^t|_{N' \cap \varOmega}$ is also continuously Fr\'{e}chet differentiable.
Therefore,
	\begin{equation*}
		t_* < t_* + \frac{T_*}{2} < t' + T_* \in S_{x_0},
	\end{equation*}
which is a contradiction.
Thus, $t_* = T_\varPhi(x_0)$ follows.

By the above steps, the conclusion is obtained.
\end{proof}


\begin{thebibliography}{99}
\addcontentsline{toc}{section}{References}

\bibitem{Arino--Hadeler--Hbid 1998} %(MR1616960)
	\newblock O. Arino, K. P. Hadeler and M. L. Hbid,
	\newblock \textit{Existence of periodic solutions for delay differential equations with state dependent delay},
	\newblock J. Differential Equations \textbf{144} (1998), 263--301.

\bibitem{Banks--Robbins--Sutton 2013} %(MR3103332)
	\newblock H. T. Banks, D. Robbins and K. L. Sutton,
	\newblock \textit{Theoretical foundations for traditional and generalized sensitivity functions
	for nonlinear delay differential equations},
	\newblock Math. Biosci. Eng. \textbf{10} (2013), 1301--1333.

\bibitem{Brezis 2011} %(MR2759829)
	\newblock H. Brezis,
	\newblock ``Functional Analysis, Sobolev Spaces and Partial Differential Equations,"
	\newblock Springer, New York, 2011.

\bibitem{Chen--Hu--Wu 2010} %(MR2610723)
	\newblock Y. Chen, Q. Hu and J. Wu,
	\newblock \textit{Second-order differentiability with respect to parameters for differential equations with adaptive delays},
	\newblock Front. Math. China \textbf{5} (2010), 221--286.

\bibitem{Chicone 2003} %(MR1970035)
	\newblock C. Chicone,
	\newblock \textit{Inertial and slow manifolds for delay equations with small delays},
	\newblock J. Differential Equations \textbf{190} (2003), 364--406.

\bibitem{Chow--Hale 1982} %(MR0660633)
	\newblock S. N. Chow and J. K. Hale,
	\newblock ``Methods of Bifurcation Theory,''
	\newblock Springer-Verlag, New York-Berlin, 1982.

\bibitem{Driver 1965} %(MR0179406)
	\newblock R. D. Driver,
	\newblock \textit{Existence and continuous dependence of solutions of a neutral functional-differential equation},
	\newblock Arch. Rational Mech. Anal. \textbf{19} (1965), 149--166.

\bibitem{Driver 1977} %(MR0477368)
	\newblock R. D. Driver,
	\newblock ``Ordinary and Delay Differential Equations,"
	\newblock Springer-Verlag, New York-Heidelberg, 1977.

\bibitem{Erneux 2009} %(MR2498700)
	\newblock T. Erneux,
	\newblock ``Applied Delay Differential Equations,"
	\newblock Springer, New York, 2009.

\bibitem{Hajek 1968} %(MR0239575)
	\newblock O. H\'{a}jek,
	\newblock \textit{Local characterisation of local semi-dynamical systems},
	\newblock Math. Systems Theory \textbf{2} (1968), 17--25.

\bibitem{Hale--Ladeira 1991} %(MR1113586)
	\newblock J. K. Hale and L. A. C. Ladeira,
	\newblock \textit{Differentiability with respect to delays},
	\newblock J. Differential Equations \textbf{92} (1991), 14--26.

\bibitem{Hale--Lunel 1993} %(MR1243878)
	\newblock J. K. Hale and S. M. Verduyn Lunel,
	\newblock ``Introduction to Functional Differential Equations,''
	\newblock Springer-Verlag, New York, 1993.

\bibitem{Hartung 2001} %(MR1975850)
	\newblock F. Hartung,
	\newblock \textit{Parameter estimation by quasilinearization in functional differential equations
	with state-dependent delays: a numerical study},
	\newblock Proceedings of the Third World Congress of Nonlinear Analysts, Part 7 (Catania, 2000).
	Nonlinear Anal. \textbf{47} (2001), 4557--4566.

\bibitem{Hartung 2013d} %(MR3138158)
	\newblock F. Hartung,
	\newblock \textit{On second-order differentiability with respect to parameters for differential equations
	with state-dependent delays},
	\newblock J. Dynam. Differential Equations \textbf{25} (2013), 1089--1138.

\bibitem{Hartung 2016b} %(MR3547449)
	\newblock F. Hartung,
	\newblock \textit{Differentiability of solutions with respect to the delay function in functional differential equations},
	\newblock Electron. J. Qual. Theory Differ. Equ. 2016, Paper No. 73, 16 pp.

\bibitem{Hartung--Krisztin--Walther--Wu 2006} %(MR2457636)
	\newblock F. Hartung, T. Krisztin, H.-O. Walther and J. Wu,
	\newblock \textit{Functional differential equations with state-dependent delays: theory and applications},
	\newblock Handbook of differential equations: ordinary differential equations. Vol. III, 435--545, Elsevier/North-Holland, Amsterdam, 2006.

\bibitem{Hartung--Turi 1997} %(MR1441270)
	\newblock F. Hartung and J. Turi,
	\newblock \textit{On differentiability of solutions with respect to parameters in state-dependent delay equations},
	\newblock J. Differential Equations \textbf{135} (1997), 192--237.

\bibitem{Kolmanovskii--Myshkis 1999} %(MR1680144)
	\newblock V. Kolmanovskii and A. Myshkis,
	\newblock ``Introduction to the Theory and Applications of Functional-Differential Equations,"
	\newblock Kluwer Academic Publishers, Dordrecht, 1999.

\bibitem{Kuznetsov 2004} %(MR2071006)
	\newblock Y. A. Kuznetsov,
	\newblock ``Elements of Applied Bifurcation Theory,"
	\newblock Third edition. Springer-Verlag, New York, 2004.

\bibitem{Louihi--Hbid--Arino 2002} %(MR1900458)
	\newblock M. Louihi, M. L. Hbid and O. Arino,
	\newblock \textit{Semigroup properties and the Crandall Liggett approximation
	for a class of differential equations with state-dependent delays},
	\newblock J. Differential Equations \textbf{181} (2002), 1--30.

\bibitem{Marsden--McCracken 1976} %(MR0494309)
	\newblock J. E. Marsden and M. McCracken,
	\newblock ``The Hopf Bifurcation and Its Applications,''
	With contributions by P. Chernoff, G. Childs, S. Chow, J. R. Dorroh, J. Guckenheimer, L. Howard, N. Kopell, O. Lanford, J. Mallet-Paret, G. Oster, O. Ruiz, S. Schecter, D. Schmidt and S. Smale.
	\newblock Springer-Verlag, New York, 1976.

\bibitem{Melvin 1972b} %(MR0340761)
	\newblock W. R. Melvin,
	\newblock \textit{A class of neutral functional differential equations},
	\newblock J. Differential Equations \textbf{12} (1972), 524--534.

\bibitem{Nishiguchi 2017}
	\newblock J. Nishiguchi,
	\newblock \textit{A necessary and sufficient condition for well-posedness of initial value problems of retarded functional differential equations},
	\newblock J. Differential Equations \textbf{263} (2017), 3491--3532.

%\bibitem{Nishiguchi preprint 2019}
%	\newblock J. Nishiguchi,
%	\newblock \textit{Some discontinuous functional differential equation
%	and its connection to smoothness of composition operators in $\L^p$},
%	\newblock submitted (arXiv:1907.03504)

\bibitem{Tao 2011} %(MR2827917)
	\newblock T. Tao,
	\newblock ``An introduction to Measure Theory,"
	\newblock American Mathematical Society, Providence, RI, 2011.

\bibitem{Walther 2003c} %(MR2019242)
	\newblock H.-O. Walther,
	\newblock \textit{The solution manifold and $C^1$-smoothness for differential equations with state-dependent delay},
	\newblock J. Differential Equations \textbf{195} (2003), 46--65.

\bibitem{Walther 2009b} %(MR2482014)
	\newblock H.-O. Walther,
	\newblock \textit{Algebraic-delay differential systems, state-dependent delay, and temporal order of reactions},
	\newblock J. Dynam. Differential Equations \textbf{21} (2009), 195--232.
\end{thebibliography}
\end{document}